\documentclass[12pt,reqno]{amsart}
\usepackage{amsaddr}
\usepackage[left=3cm,top=2cm,right=3cm,bottom=2cm]{geometry}
\usepackage{amsbsy}
\usepackage{adjustbox}
\usepackage{epsfig}
\usepackage{tikz}
\usepackage{amsmath,amsfonts}
\usepackage{amssymb,amsthm}
\usepackage{graphicx}
\usepackage{graphics}
\usepackage{comment}
\usepackage[english]{babel}
\usepackage{amscd}
\usepackage{latexsym}
\usepackage{mathrsfs}
\usepackage{booktabs}
\usepackage{xcolor}
\usepackage{enumitem}
\usepackage{multicol}
\usepackage{multirow}
\usepackage{blkarray}
\usepackage{mathtools}
\usepackage{fancyhdr}
\usepackage{lipsum}
\usepackage[colorinlistoftodos]{todonotes}
\usepackage[colorlinks=true, allcolors=blue]{hyperref}
\usepackage{float}
\usepackage[thinlines]{easytable}
\usetikzlibrary{calc}
\usepackage{caption}
\usepackage{environ}
\usetikzlibrary{shapes.geometric}
\usepackage{fourier} 
\usepackage{array}
\usepackage{makecell}
\usepackage{siunitx}
\usetikzlibrary{decorations.markings}
\tikzstyle{vertex}=[circle, draw, inner sep=0pt, minimum size=5pt]
\newcommand{\vertex}{\node[vertex]}

\newtheorem{teo}{Theorem}[section]

\newtheorem{lema}[teo]{Lemma}

\newtheorem{prop}[teo]{Proposition}

\newtheorem{prob}{Problem}

\newtheorem{exemplo}[teo]{Example}

\renewcommand{\qedsymbol}{$\square$}
\newcommand*{\QEDA}{\hfill\ensuremath{\square}}%
\hypersetup{colorlinks=true,linkcolor=blue,citecolor=red}

\usepackage{xparse}
\let\oldexemplo\exemplo
\RenewDocumentCommand{\exemplo}{o}{%
  \IfNoValueTF{#1}
    {\oldexemplo}
    {\oldexemplo[#1]}%
  \normalfont
}

\begin{document}
%\begin{frontmatter}

\title[]{Graphs with at most two nonzero distinct absolute eigenvalues}
\subjclass{05C50, 15A18, 15A60}
\keywords{graph energy, trace norm, few eigenvalues}

\author[]{N. E. Ar\'evalo\textsuperscript{1}, R. O. Braga\textsuperscript{2} \and V. M. Rodrigues\textsuperscript{3}}
\email{\textsuperscript{1}nearevalob@unal.edu.co}
%\author[]{R. O. Braga\textsuperscript{2}}
\email{\textsuperscript{2}rbraga@ufrgs.br}
%\author[]{V. M. Rodrigues\textsuperscript{3}}
\email{\textsuperscript{3}vrodrig@mat.ufrgs.br}
\address{Instituto de Matem\'atica, Universidade Federal do Rio Grande do Sul, Porto Alegre, Brazil}

%\author[1] {N. E. Ar\'evalo}\ead{nearevalob@unal.edu.co}

%\author[1] {R. O. Braga}\ead{rbraga@ufrgs.br}

%\author[1] {V. M. Rodrigues\corref{cor1}}\ead{vrodrig@mat.ufrgs.br}


%\address[1] {Instituto de Matem\'atica e Estat\'istica, Universidade Federal do Rio Grande do Sul, Porto Alegre, Brazil}

\begin{abstract}
In his survey ``Beyond graph energy: Norms of graphs and matrices" \cite{Niki2}, Nikiforov proposed two problems concerning characterizing the graphs that attain equality in a lower bound and in a upper bound for the energy of a graph, respectively. We show that these graphs have at most two nonzero distinct absolute eigenvalues and investigate the proposed problems organizing our study according to the type of spectrum they can have. In most cases all graphs are characterized. Infinite families of graphs are given otherwise. We also show that all graphs satifying the properties required in the problems are integral, except for complete bipartite graphs $K_{p,q}$ and disconnected graphs with a connected component $K_{p,q}$, where $pq$ is not a perfect square.
\end{abstract}

%\begin{keyword}
%graph energy \sep trace norm \sep few eigenvalues

%\MSC 05C50 \sep 15A18  \sep 15A60
%\end{keyword}
\subjclass{05C50, 15A18, 15A60}
\keywords{graph energy, trace norm, few eigenvalues}

%\end{frontmatter}
\maketitle

%%%%%%%%%%%%%%%%%%%%%%%%%%%%%%%%%%%%%%%%%%%%%%%%%%%%%%%%%%%%%%%%%%%%%%%%%%%%%%%%%%
%%%%%%%%%%%%%%%%%%%%%%%%%%%%%%%%%%%%%%%%%%%%%%%%%%%%%%%%%%%%%%%%%%%%%%%%%%%%%%%%%%

\section{Introduction}\label{sec:I} 

Among the various spectral parameters studied in spectral graph theory, one can highlight the \textit{energy} of a graph, introduced by Gutman in 1978 \cite{Gutman}. The energy of a graph $G$ with $n$ vertices is defined as
$$\mathcal{E}\left(G\right)=\sum_{j=1}^{n}\left|\lambda_{j}\right|,$$
where $\lambda_1\geq\lambda_2\geq\cdots\geq\lambda_n$ are the eigenvalues of the adjacency matrix of $G$. Throughout this paper $G$ is a simple nonoriented graph  with  adjacency matrix $A=A\left(G\right)$. The eigenvalues of $G$ are the eigenvalues of $A$ and the collection of eigenvalues of $G$ is called the \textit{spectrum of} $G$, denoted by $Spec\left(G\right)$. 

Graph energy has been intensively studied; for a thorough introduction to the subject see \cite{Li}.  One of the questions of great interest in spectral graph theory is which graphs (of a given class) have the largest or smallest energy values. There is also a great effort in obtaining effective bounds on graph energy. A step further in this study was taken by Nikiforov \cite{Niki1} who generalized  the concept of graph energy  by defining the energy of any matrix with complex entries. 

Recall that the \textit{singular values} of a complex matrix $M$ of order $m \times n$ are the square roots of the eigenvalues of $M^{*}M$, where $M^{*}$ is the conjugate transpose of $M$. The \textit{trace norm} $\left\|M\right\|_{*}$ of $M$ is the sum of its singular values, which for a real symmetric matrix are exactly the modules of its eigenvalues. Therefore,  the trace norm of the adjacency matrix  of a graph $G$ is the energy of $G$, that is $\left\|A\left(G\right)\right\|_{*} = \mathcal{E}\left(G\right)$. This observation made by Nikiforov in  \cite{Niki1} ``triggered some sort of a chain reaction", as Nikiforov himself noted in his survey on norms of graphs and matrices \cite{Niki2}. Besides extending the concept of graph energy  to non-symmetric and even to non-square matrices, research on matrix norms provide new techniques to the study of graph energy. For instance,  matrix norms were used in \cite{Niki2} to derive and extend a well-known upper bound on the energy of a graph of order $n$ given by Koolen and Moulton in \cite{Koolen}, namely 
\begin{equation}
\label{km}
\mathcal{E}\left(G\right) \leq \frac{n}{2}\left(1+\sqrt{n}\right). 
\end{equation}

There are several bounds for the trace norm of general matrices. A lower bound on the trace norm of a complex matrix $M=\left[a_{ij}\right]$ of order $m \times n$ with rank at least $2$ was given by Nikiforov in \cite{Niki1}. He showed that
	\begin{equation}\label{eqn101}
	\left\|M\right\|_{*}\geq\sigma_{1}\left(M\right)+\frac{1}{\sigma_{2}\left(M\right)} \left(\sum_{i,j}\left|a_{ij}\right|^2-\sigma_{1}^2\left(M\right)\right),
	\end{equation}
	where $\sigma_1\geq\sigma_2\geq\cdots\geq\sigma_n$ are the singular values of $M$. 
Note that inequality (\ref{eqn101}) gives a lower bound for the energy of a \textit{nonempty graph}, i.e. a graph with at least one edge.  In fact in this case equality holds if and only if all nonzero eigenvalues of $G$ other than the \textit{index} (i.e., the largest eigenvalue of $G$) have the same absolute value. Nikiforov observed in \cite{Niki2} that  bound (\ref{eqn101}) is quite efficient for graphs:  equality is attained for the adjacency matrix of a design graph, or the complete graph $K_n$, or a complete bipartite graph. He then proposed the following problem: 

\begin{prob}\cite[Problem 2.13]{Niki2}\label{prob}
Give a constructive characterization of all graphs $G$ such that the nonzero eigenvalues of $G$ other than its index have the same absolute value.
\end{prob}

A slightly different problem was also proposed in \cite{Niki2}. Nikiforov noted that bound (\ref{km})  was derived  by Koulen and Moulen from an upper bound for the energy of a graph $G$ of order $n$ with $m$ edges and index $\lambda_1$, namely
\begin{equation}
\label{km2}
\mathcal{E}\left(G\right)\leq\lambda_1 +\sqrt{\left(n-1\right) \left(2m-\lambda_1^2\right) }, 
\end{equation}
with equality if and only if $\left|\lambda_2\right|=\cdots=\left|\lambda_n\right|$. He observed that if equality is attained in (\ref{km2}) and $G$ is regular, then  either $G=\left(n/2\right)K_2$, or $G=K_n$, or $G$ is a design graph. Besides if equality is attained and $G$ is not regular and is disconnected, then $G = K_{n-2r}+rK_2$. This motivated the following problem:

\begin{prob}\cite[Problem 2.40]{Niki2}\label{prob2}
Give a constructive characterization of all irregular connected graphs $G$ of order $n$ with $\left|\lambda_2\left(G\right)\right|=\cdots=\left|\lambda_n\left(G\right)\right|$.
\end{prob}

In this work we investigate Problems \ref{prob} and \ref{prob2}.  Let $\mathcal{G}_n$ and $\mathcal{H}_n$ be the classes of nonempty graphs of order $n$ that satisfy the properties required in Problems \ref{prob} and \ref{prob2}, respectively. Clearly the graphs in $\mathcal{G}_n$ have at most two nonzero distinct absolute eigenvalues and $\mathcal{H}_n$ is a subset of $\mathcal{G}_n$. It is known that a graph $G$ of order $n$ has only one eigenvalue if and only if it is the empty graph, and $G$ has exactly two distinct eigenvalues if and only if $G=\left(n/r\right)K_r$. Since $K_n$ is a regular graph with $Spec\left(K_n\right)=\left\lbrace\left[-1\right]^{n-1},\left[n-1\right]^{1}\right\rbrace$, every graph with $2$ distinct eigenvalues belongs to $\mathcal{G}_n$ but not to $\mathcal{H}_n$. Hence it remains to characterize graphs with $3$ or $4$  distinct eigenvalues in the family $\mathcal{G}_n$, and in particular those with $3$ distinct eigenvalues that belong to $\mathcal{H}_n$.

Graphs with few distinct eigenvalues form a largely studied class of graphs. They were first investigated by Doob  in 1970 \cite{Doob} and have been studied by several authors since then. A first nontrivial family of such graphs that have received a great deal of attention are the \textit{strongly regular graphs}, which are the regular graphs with exactly three distinct eigenvalues \cite{Shrikhande}. Examples, constructions, characterizations and some nonexistence results about connected regular graphs with four eigenvalues are given in \cite{Huang}, \cite{VanDam1995} and \cite{VanDam4}. However, there are many open questions about irregular graphs with three or four eigenvalues. Examples of irregular graphs with three distinct eigenvalues are given in \cite{Chuang} and \cite{MUZY}. All irregular graphs with three eigenvalues with least eigenvalue $-2$ are determined in \cite{VANDAM}, where new infinite families of examples are also given. 

In our study we first consider connected graphs. The Perron-Frobenius theorem (see for instance \cite{Horn}) asserts that the index of a connected graph is simple. Thus the only possible spectra of a connected graph $G \in \mathcal{G}_n$ with more than $2$ eigenvalues are:

\begin{itemize}
\item[] \textbf{Case 1 -} Three distinct eigenvalues:

\begin{itemize}
\item[(a)] $\,\,Spec\left(G\right)=\left\lbrace\left[\lambda\right]^1,\left[0\right]^{n-t-1},\left[\mu\right]^t\right\rbrace$, $\lambda>0>\mu$
\end{itemize}
or	
\begin{itemize}
\item[(b)] $\,\,Spec\left(G\right)=\left\lbrace\left[\lambda\right]^1,\left[\mu\right]^{n-t-1},\left[-\mu\right]^t\right\rbrace$, $\lambda>\mu>0.$
\end{itemize}
\vspace{2mm}
\item [] \textbf{Case 2 -} Four distinct eigenvalues:
\begin{itemize}
\item[] $\,\,Spec\left(G\right)=\left\lbrace\left[\lambda\right]^1,\left[\mu\right]^{n-k-t-1},\left[0\right]^{t},\left[-\mu\right]^k\right\rbrace$, $\lambda>\mu>0.$
\end{itemize}
\end{itemize}
Note that the graphs in $\mathcal{H}_n$ with more than $2$ distinct eigenvalues have spectrum of the form given in Case 1(b).

In Section \ref{sec:II-I} we study graphs with spectrum as in Case 1.  Graphs of Case 2 are investigated in Section \ref{sec:III-I}. Table \ref{table} summarizes the results we obtained in these cases, that is, our contribution to the solution of Problem \ref{prob} for connected graphs with more than $2$ eigenvalues. Our contribution to Problem \ref{prob2} is given by the irregular graphs in the second row of Table \ref{table}. In Section \ref{sec4} we investigate disconnected graphs in $\mathcal{G}_n$. Final remarks are made in 
Section \ref{open}.

\begin{table}[H]
\def\arraystretch{1.3} % vertical stretch factor
\centering
\begin{adjustbox}{max width=\textwidth}
\begin{tabular}{|c|c|c|}% 
     \hline \textit{Spectrum} &  \textit{Regular graphs} & \textit{Irregular graphs} \\
      \hline
      \begin{tabular}{l}
     Three distinct eigenvalues, \\one equals $0$:  \\ \hspace{0.3cm}
         $\left\lbrace\left[\lambda\right]^1,\left[0\right]^{n-t-1},\left[\mu\right]^t\right\rbrace$, \\  $\lambda > 0 > \mu$.
    \end{tabular}  
        & 
        \begin{tabular}{l} Integral complete $r$-partite graph \\ with all parts of size $-\mu$, with $r \geq 2$  \\ \hspace{1cm} (Theorem \ref{teoimp2}) \end{tabular} 
        & \begin{tabular}{l} Complete bipartite graph $K_{p,q}$, \\ with $p \neq q$ and $pq =\mu^2$ \\  \hspace{1.2cm}
        (Theorem \ref{teoimp2})  \end{tabular}\\
      \hline
      \multirow{7}{*}{\begin{tabular}{l}
      Three distinct nonzero\\ eigenvalues:\\\hspace{0.3cm} $\left\lbrace\left[\lambda\right]^1,\left[\mu\right]^{n-t-1},\left[-\mu\right]^t\right\rbrace$, \\ $\lambda > \mu > 0$.
     \end{tabular}} &  \multirow{6}{*}{\begin{tabular}{l} Design graph with parameters\\  $\left(n,\lambda,\lambda-\mu^2,\lambda-\mu^2\right)$\\ \hspace{1cm} (Theorem \ref{teoimp3}) \end{tabular}}    
     &  \begin{tabular}{l}
     Graphs with $\mu=2$: \\ 
     - cone over the Shrikhande graph \\ 
     - cone over the lattice graph $L_{2}\left(4\right)$\\
     - graph on the points of $AG\left(3,2\right)$ \\ 
     \hspace{1.5cm} (Theorem \ref{teoimp3}) \end{tabular} \\ 
     \cline{3-3}
     & & \begin{tabular}{l}Graphs in the families of \\
     Examples \ref{ex1} and \ref{ex2novo} \end{tabular}\\
     \cline{3-3}
     %Graphs in Examples \ref{ex1} and \ref{ex2novo} 
     & &\begin{tabular}{l}
     Integral multiplicative graphs with \\
         $\mu \geq 3$  and $n > 30$, that are not in the \\ families  of  Examples \ref{ex1} and \ref{ex2novo} \\ \hspace{2cm} (Open)
       \end{tabular}\\
      \hline
      \multirow{7}{*}{
      \begin{tabular}{l}
 Four distinct  eigenvalues, \\one equals $0$: \\ $\left\lbrace\left[\lambda\right]^1,\left[\mu\right]^{n-k-t-1},\left[0\right]^{t},\left[-\mu\right]^k\right\rbrace$, \\ $\lambda>\mu>0.$ \end{tabular}} 
 & 
  \begin{tabular}{l} $\overline{Q_3\circledast J_{\frac{\mu}{2}}}$ with $\lambda = 2\mu$ \\ \hspace{0.5cm}(Theorem \ref{teoimp4}) \end{tabular}
  
 & \begin{tabular}{l} \\  Graphs in the families of \\Examples \ref{fam1irr4} and \ref{fam2irr4}
 \end{tabular}
     \\ \cline{2-2}
  & 
  \begin{tabular}{l}
  Graphs in the families of \\Examples \ref{reg4srg} and \ref{reg4van}
  \end{tabular}
  &
  \\ \cline{2-3}
  & 
   \begin{tabular}{l}  Integral graphs with $n>30$, where \\only  the index is  simple, that  are not in \\the families of Examples \ref{reg4srg} and \ref{reg4van} \\ \hspace{2cm} (Open) \end{tabular} 
   
      & 
      \begin{tabular}{l}
      Graphs not in the families \\ of Examples \ref{fam1irr4} and \ref{fam2irr4}\\ \hspace{1cm}(Open)
      \end{tabular}
      \\ \hline
    \end{tabular}
\end{adjustbox}
    \caption{Connected graphs in $\mathcal{G}_n$ }
    \label{table}
\end{table}

We completely characterize the graphs in $\mathcal{H}_n$, which solves Problem \ref{prob2} except that the characterization we give is not always constructive. In that case we present two infinite families of graphs. Similarly, for the  graphs in $\mathcal{G}_n$ with four distinct eigenvalues, a partial constructive characterization in the case they are regular and infinite families satisfying the remaining cases are given.

It follows from this work that except for the complete bipartite graphs $K_{p,q}$ and disconnected graphs with a connected component $K_{p,q}$, where $pq$ is not a perfect square, all graphs in $\mathcal{G}_n$ are integral, i.e. their spectra consists entirely of integers. In addition, since the line graph of a regular integral graph is also integral \cite{Balinska}, other integral graphs can be obtained by taking the line graphs of the graphs in $\mathcal{G}_n$. 

%%%%%%%%%%%%%%%%%%%%%%%%%%%%%%%%%%%%%%%%%%%%%%%%%%%%%%%%%%%%%%%%%%%%%%%%%%
%%%%%%%%%%%%%%%%%%%%%%%%%%%%%%%%%%%%%%%%%%%%%%%%%%%%%%%%%%%%%%%%%%%%%%%%%
\section{Three distinct eigenvalues}\label{sec:II-I}

In this section we characterize all connected graphs in the family $\mathcal{G}_n$ that have exactly three distinct eigenvalues. 

\subsection{0 is an eigenvalue}
\label{subsecao1}

\begin{sloppypar}
We first consider connected graphs in $\mathcal{G}_n$ with spectrum             ${\left\lbrace\left[\lambda\right]^1,\left[0\right]^{n-t-1},\left[\mu\right]^t\right\rbrace}$, where $1 \leq t \leq n-2$ and $\lambda>0>\mu$.
We need some auxiliary results and definitions. 
Recall that an \textit{$r$-partite graph} is a graph whose vertices can be partitioned into $r$ disjoint sets, called \textit{parts}, such that no two vertices within the same part are adjacent. We write $K_{p_1,p_2,\ldots,p_r}$ to represent the $r$-partite graph with parts of sizes $p_1\leq p_2 \leq\ldots\leq p_r$.  A \textit{complete $r$-partite graph} is an $r$-partite graph such that every two vertices of different sets of the partition are adjacent. When $r=2$ we have a \textit{bipartite graph} in the former and a  \textit{complete bipartite graph} in the latter case.
\end{sloppypar}

\begin{lema}\cite[Theorem 2.3.4]{Asratian}\label{lemanovo}
A graph with at least one edge is bipartite if and only if its spectrum is symmetric with respect to $0$.
\end{lema}

\begin{lema}\label{teo8}
Let $G$ be a connected graph with three distinct eigenvalues. If $G$ is bipartite or its index is not an integer, then $G$ is a complete bipartite graph.
\end{lema}

\begin{proof}\renewcommand{\qedsymbol}{}
The case where the index of $G$ is not an integer is a result by Van Dam \cite[Proposition 2]{VANDAM}. It is known that the diameter of a connected graph is strictly less than the number of its distinct eigenvalues (see for instance \cite{Brouwer}).  Hence the diameter of $G$ is at most $2$.  
Suppose that $G$ is bipartite with parts $U$ and $W$. If $G$ is not a complete bipartite graph then there is a vertex $u\in U$ and a vertex $w\in W$ that are not adjacent. Thus since $G$ is connected, the distance between $u$ and $v$ is at least $3$, a contradiction. \QEDA

\end{proof}

\begin{lema}\cite[Theorem 7.4]{Beineke}\label{teo1}
A connected graph $G$ is a complete $r$-partite graph if and only if $\lambda_2\leq 0$, where $\lambda_2$ is the second largest eigenvalue of $G$.
\end{lema}

\begin{lema}\label{lemacom}
Let $G=K_{p_1,p_2,\ldots,p_r}$ be the complete $r$-partite graph of order $n$.
\begin{enumerate}
\item[$(i)$] \cite[Lemma 2]{Esser}
The characteristic polynomial of $G$ can be written as
\begin{equation*}
P_G\left(x\right) =x^{n-r}\left[\prod_{i=1}^{r}\left(x+p_i\right) -\sum_{i=1}^{r}p_i\prod_{\substack{j=1\\j\neq i}}^{r}\left(x+p_j\right)\right].
\end{equation*}
\item[$(ii)$] \cite[Theorem 1]{Esser} The $r-1$ negative eigenvalues $\lambda_{n-r+2},\ldots,\lambda_n$ of $G$ satisfy the ine\-qualities
\begin{equation*}
p_1\leq -\lambda_{n-r+2}\leq p_2\leq -\lambda_{n-r+3}\leq p_3\leq\cdots\leq p_{r-1}\leq -\lambda_n\leq p_r.
\end{equation*}
\end{enumerate}
\end{lema}

The result below gives a constructive characterization of all graphs in the family $\mathcal{G}_n$ that have three distinct eigenvalues,  one  equals zero.

\begin{teo}\label{teoimp2}
Let $G$ be a connected graph of order $n$ with spectrum  $\left\lbrace\left[\lambda\right]^1,\left[0\right]^{n-t-1},\left[\mu\right]^t\right\rbrace$, where $1 \leq t \leq n-2$ and $\lambda>0>\mu$. Then 
\begin{itemize}
\item [(i)] $\mu=-\lambda$ if and only if $G$ is a complete bipartite graph.
\item [(ii)] $\mu\neq-\lambda$ if and only if $G$ is an integral complete $\left(t+1\right)$-partite graph with all parts of size $-\mu$, with $t\geq 2$.
\end{itemize}
\end{teo}

\begin{proof}\renewcommand{\qedsymbol}{}
It follows from Lemma \ref{teo1} that $G$ is a complete multipartite graph. 

If $\mu=-\lambda$ then $G$ is bipartite by Lemma \ref{lemanovo}, and the converse follows from the fact that the spectrum of the complete bipartite graph $K_{p,q}$ is $\Big\{\left[\sqrt{pq}\right]^1, \left[0 \right]^{p+q-2},\left[-\sqrt{pq}\right]^1\Big\}$. 

Suppose that $t\geq 2$ and  $G$ is an integral complete $\left(t+1\right)$-partite graph with all parts of size $-\mu$.  Then by Lemma \ref{lemacom} the characteristic polynomial of $G$ is
\begin{align*}
P_G\left(x\right)&=x^{n-(t+1)}\left[\prod_{i=1}^{t+1}\left(x-\mu\right) -\sum_{i=1}^{t+1}\left(-\mu\right)\prod_{\substack{j=1\\j\neq i}}^{t+1}\left(x-\mu\right) \right]\\
&=x^{n-t-1}\left[\left(x-\mu\right)^{t+1}+\left(t+1\right) \mu\left(x-\mu\right)^{t}\right]\\
&=x^{n-t-1}\left(x-\mu\right)^{t}\left(x+t\mu\right).
\end{align*}

Hence $Spec\left(G\right)=\left\lbrace \left[-t\mu\right]^1,\left[0\right]^{n-t-1},\left[\mu\right]^{t}\right\rbrace$ and so $-\lambda = t\mu \neq \mu$. 

Conversely, since the eigenvalues add up to $0$ we have $\lambda+t\mu=0$. Thus if  $\mu\neq -\lambda$ then \mbox{$t\geq 2$} and, by Lemma \ref{lemanovo}, $G$ is not bipartite. It follows from Lemma \ref{teo8} that $\lambda\in\mathbb{Z}$, consequently $\mu\in\mathbb{Q}$. The rational root theorem implies that $\mu\in\mathbb{Z}$, as the characteristic polynomial of $G$ is monic with integral coefficients. Then $G$ is integral. Besides, since the multiplicity of $0$ as an eigenvalue of $G$ is $n-\left(t+1\right)$, by Lemma \ref{lemacom} we conclude that $G$ has $r=t+1$ parts. The $r-1$ negative eigenvalues of $G$ are all equal to $\mu$, hence by  Lemma \ref{lemacom} we have $$p_1\leq -\mu\leq p_2\leq -\mu\leq p_3\leq\cdots\leq p_{t}\leq -\mu\leq p_{t+1},$$
where $p_1 \leq p_2 \leq \dots \leq p_{t+1}$ are the sizes of the parts of $G$. Therefore $p_2=p_3=\cdots=p_{t}=-\mu$ and $p_1\leq -\mu\leq p_{t+1}$. Lemma \ref{lemacom} also implies that the characteristic polynomial of $G$ is
\begin{align*}
P_G\left(x\right)
= & x^{n-t-1}\left[\left(x+p_1\right)\left(x-\mu\right)^{t-1}\left(x+p_{t+1}\right)-p_1\left(x-\mu\right)^{t-1}\left(x+p_{t+1}\right)\right.\\
&\left.+\left(t-1\right)\mu\left(x+p_1\right)\left(x-\mu\right)^{t-2}\left(x+p_{t+1}\right)-p_{t+1}\left(x+p_1\right)\left(x-\mu\right)^{t-1}\right]\\
= & x^{n-t-1}\left(x-\mu\right)^{t-2}\left[x^3+\left(t-2\right)\mu x^2+\left((t-1)\mu\left(p_1+p_{t+1}\right)-p_1p_{t+1}\right)x\right.\\
& \left.+t\mu p_1p_{t+1}\right].
\end{align*}

On the other hand, from the spectrum of $G$ and the fact that $\lambda = -t\mu$ we have
\begin{align*}
P_G\left(x\right) & =  x^{n-t-1}\left(x+t\mu\right) \left(x-\mu\right)^t\\
   & =  x^{n-t-1}\left(x-\mu\right)^{t-2}\left[x^3+\left(t-2\right)\mu x^2+\left(1-2t\right)\mu^2 x+t\mu^3\right].
\end{align*}

Hence
\begin{equation}\label{eqn11}
\left(t-1\right)\mu\left(p_1+p_{t+1}\right)-p_1p_{t+1}=\left(1-2t\right)\mu^2
\end{equation}
and
\begin{equation}\label{eqn10}
t\mu p_1p_{t+1}=t\mu^3.
\end{equation}  			

From (\ref{eqn10}) we get $p_1p_{t+1}=\mu^2$. Replacing that into (\ref{eqn11}) we obtain $p_1+p_{t+1}=-2\mu$. It follows that $p_{1}=p_{t+1}=-\mu$. Then $G$ is a complete multipartite integral graph with $t+1$ parts of size $-\mu$. \QEDA

\end{proof}

%%%%%%%%%%%%%%%%%%%%%%%%%%%%%%%%%%%%%%%%%%%%%%%%%%%%%%%%%%%%%%%%%%%%%%%%%%%%%%%%%%

\subsection{0 is not an eigenvalue}\label{subsecao2}

We now consider connected graphs in $\mathcal{G}_n$ with spectrum $\left\lbrace\left[\lambda\right]^1,\left[\mu\right]^{n-t-1},\left[-\mu\right]^t\right\rbrace$, where $ 1\leq t \leq n-2$ and $\lambda>\mu>0$. Recall that all graphs in the family $\mathcal{H}_n \subset \mathcal{G}_n$ have spectrum of this form. Again we start with some auxiliary results and definitions.

A graph is \textit{regular} if all its vertices have the same degree. A regular graph with vertices of degree $r$ is called an \textit{$r$-regular graph}. It is well known that if $G$ is an $r$-regular graph then $r$ is the largest eigenvalue of $G$ and its multiplicity is equal to the number of connected components of $G$. A graph of order $n$ is called \textit{strongly regular} with parameters $\left(n,r,\alpha,\beta\right)$ (or a $srg\left(n,r,\alpha,\beta\right)$ for short) when it is $r$-regular, any two vertices have exactly $\alpha$ or $\beta$ common neighbors depending whether they are adjacent or nonadjacent, and the graph is neither complete, nor empty. A $srg\left(n,r,\alpha,\alpha\right)$ is often called a \textit{design graph}. 

\begin{lema}\label{lemasrg} \cite[Corollary 3.4.11]{Stani}
If $G$ is a strongly regular graph with parameters $(n, r,\alpha, \beta)$ then 
 $Spec\left(G\right) = \left\{[r]^1, [\lambda_2]^{m_2}, [\lambda_3]^{m_3}\right\}$, where 
 
 \begin{center}
$\lambda_2 = \frac{\left(\alpha-\beta\right)+\sqrt{\left(\alpha-\beta\right)^2+4\left(r-\beta\right)}}{2} $\hspace{0.25cm}, \hspace{0.25cm} $\lambda_3 =  \frac{\left(\alpha-\beta\right)-\sqrt{\left(\alpha-\beta\right)^2+4\left(r-\beta\right)}}{2}$,

\vspace{0.25cm}
$m_2 =\frac{1}{2}\left(n-1-\frac{2r+\left(n-1\right)\left(\alpha-\beta\right)}{\sqrt{\left(\alpha-\beta\right)^2+4\left(r-\beta\right)}}\right)$ \hspace{0.25cm} and \hspace{0.25cm} $m_3=\frac{1}{2}\left(n-1+\frac{2r+\left(n-1\right)\left(\alpha-\beta\right)}{\sqrt{\left(\alpha-\beta\right)^2+4\left(r-\beta\right)}}\right)$.
 \end{center}
\end{lema}

\begin{exemplo}\label{Shrikhande}
The Shrikhande graph depicted in Figure \ref{grashik} is a $srg\left(16,6,2,2\right)$ with spectrum $\left\lbrace \left[6\right]^1,\left[2\right]^6,\left[-2\right]^9\right\rbrace$.
\begin{figure}[H]
\[\begin{tikzpicture}[scale=0.8]
\vertex[fill] (1) at (-2,2){};
\vertex[fill] (2) at (2,2) {};
\vertex[fill] (3) at (0,1.5) {};
\vertex[fill] (4) at (0,1) {};
\vertex[fill] (5) at (-0.5,0.5) {};
\vertex[fill] (6) at (0.5,0.5) {};
\vertex[fill] (7) at (-1.5,0) {};
\vertex[fill] (8) at (-1,0) {};
\vertex[fill] (9) at (1,0) {};
\vertex[fill] (10) at (1.5,0) {};
\vertex[fill] (11) at (-0.5,-0.5) {};
\vertex[fill] (12) at (0.5,-0.5) {};
\vertex[fill] (13) at (0,-1) {};
\vertex[fill] (14) at (0,-1.5) {};
\vertex[fill] (15) at (-2,-2) {};
\vertex[fill] (16) at (2,-2) {};
\path
(1) edge (2)
(1) edge (3)
(1) edge (4)
(1) edge (8)
(1) edge (7)
(1) edge (15)
(2) edge (3)
(2) edge (4)
(2) edge (9)
(2) edge (10)
(2) edge (16)
(15) edge (7)
(15) edge (8)
(15) edge (13)
(15) edge (14)
(15) edge (16)
(16) edge (9)
(16) edge (10)
(16) edge (13)
(16) edge (14)
(3) edge (8)
(3) edge (5)
(3) edge (6)
(3) edge (9)
(7) edge (4)
(7) edge (5)
(7) edge (11)
(7) edge (13)
(14) edge (11)
(14) edge (8)
(14) edge (9)
(14) edge (12)
(10) edge (13)
(10) edge (12)
(10) edge (4)
(10) edge (6)
(4) edge (11)
(4) edge (12)
(8) edge (6)
(8) edge (12)
(13) edge (5)
(13) edge (6)
(9) edge (5)
(9) edge (11)
(5) edge (6)
(5) edge (11)
(12) edge (6)
(12) edge (11)
;
\end{tikzpicture}\]
\captionof{figure}{Shrikhande graph}
\label{grashik}
\end{figure}
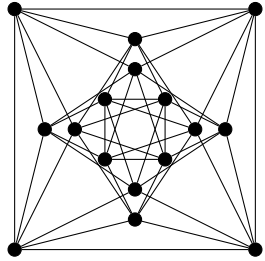
\end{exemplo}

The \textit{line graph} of a graph $G$, denoted by $L(G)$, is a graph such that each vertex represents an edge of $G$ and two vertices are adjacent if and only if their corresponding edges have an endpoint in common in $G$.  If $G$ is $r$-regular, then  $L\left(G\right)$ is $(2r-2)$-regular (see for instance \cite{Brouwer}). Besides, it is observed in \cite{Balinska} that if $G$ is integral, then $L(G)$ is integral, since its characteristic polynomial can be expressed as $ P_{L\left(G\right)}\left(x\right)=\left(x+2\right)^{m-n}P_{G}\left(x-r+2\right)$,  where $m=\left(nr/2\right)$.

\begin{exemplo}
The line graph of the complete bipartite graph $K_{4,4}$,shown in Figure \ref{graphl24}, is a $srg\left(16,6,2,2\right)$ with spectrum $\left\lbrace \left[6\right]^1,\left[2\right]^6,\left[-2\right]^9\right\rbrace$. It is usually called the lattice graph of order $4$ and denoted by $L_2\left(4\right)$. Note that though the Shrikhande graph and the lattice graph $L_2\left(4\right)$ have the same parameters, they are not isomorphic. Indeed, the Shrikhande graph has cycles of length 3 while $L_2\left(4\right)$ does not. These are the only designs with those parameters.

\begin{figure}[H]\centering
\[\begin{tikzpicture}[scale=0.8]
\vertex[fill] (1) at (0,0) {};
\vertex[fill] (2) at (1,0) {};
\vertex[fill] (3) at (2,0) {};
\vertex[fill] (4) at (3,0) {};
\vertex[fill] (5) at (0,-1) {};
\vertex[fill] (6) at (1,-1) {};
\vertex[fill] (7) at (2,-1) {};
\vertex[fill] (8) at (3,-1) {};
\vertex[fill] (9) at (0,-2) {};
\vertex[fill] (10) at (1,-2) {};
\vertex[fill] (11) at (2,-2) {};
\vertex[fill] (12) at (3,-2) {};
\vertex[fill] (13) at (0,-3) {};
\vertex[fill] (14) at (1,-3) {};
\vertex[fill] (15) at (2,-3) {};
\vertex[fill] (16) at (3,-3) {};
\path
(1) edge (2)
(1) edge (5)
(2) edge (3)
(2) edge (6)
(3) edge (4)
(3) edge (7)
(4) edge (8)
(5) edge (6)
(5) edge (9)
(6) edge (7)
(6) edge (10)
(7) edge (8)
(7) edge (11)
(8) edge (12)
(9) edge (10)
(9) edge (13)
(10) edge (11)
(10) edge (14)
(11) edge (12)
(11) edge (15)
(12) edge (16)
(13) edge (14)
(14) edge (15)
(15) edge (16)
(3) edge[bend right=20] (1)
(4) edge[bend right=20] (2)
(4) edge[bend right=20] (1)
(1) edge[bend right=20] (9)
(1) edge[bend right=20] (13)
(5) edge[bend right=20] (13)
(3) edge[bend right=20] (1)
(3) edge[bend right=20] (1)
(3) edge[bend right=20] (1)
(12) edge[bend right=20] (4)
(16) edge[bend right=20] (4)
(16) edge[bend right=20] (8)
(13) edge[bend right=20] (15)
(13) edge[bend right=20] (16)
(14) edge[bend right=20] (16)
(7) edge[bend right=20] (5)
(8) edge[bend right=20] (6)
(8) edge[bend right=20] (5)
(9) edge[bend right=20] (11)
(10) edge[bend right=20] (12)
(9) edge[bend right=20] (12)
(2) edge[bend right=20] (10)
(2) edge[bend right=20] (14)
(6) edge[bend right=20] (14)
(11) edge[bend right=20] (3)
(15) edge[bend right=20] (3)
(15) edge[bend right=20] (7)
;
\end{tikzpicture}\]
\captionof{figure}{$L_{2}\left(4\right)$ graph}
\label{graphl24}
\end{figure}
\end{exemplo}

\begin{lema}\cite{Shrikhande}\label{treesrg} 
A connected regular graph $G$ is strongly regular if and only if it has exactly three distinct eigenvalues.
\end{lema}

Recall that a \textit{cone over} a graph $G$ is the graph obtained by adding a vertex to $G$ and connecting this vertex  to all vertices of $G$. 

\begin{exemplo} \label{cone}
The cone over the Shrikhande graph and the cone over the lattice graph $L_2\left(4\right)$ are irregular, nonisomorphic and cospectral, with spectrum $\left\lbrace \left[8\right]^1,\left[2\right]^6,\left[-2\right]^{10}\right\rbrace$. (See also Example \ref{ex1})
\end{exemplo}

Given a set $X$ with $n$ elements, called \textit{points}, and integers $b, k, r, \alpha \geq 1$, a \textit{balanced incomplete block design} (\textit{BIBD} or $2$-\textit{design}) is a family of $b$ subsets of $X$, called \textit{blocks}, such that each element of $X$ is contained in $r$ blocks, each block contains $k$ elements, and each pair of elements is simultaneously contained in $\alpha$ blocks. In the particular case $b=n$ (or equivalently $r=k$),  the design is called \textit{symmetric} with parameters $\left(n,r,\alpha\right)$.

A \textit{multiplicative design}, as defined by Ryser \cite{Ryser},
is a family of $n$ subsets of an $n$-set, \mbox{$n\geq3$}, such that the $\left(0,1\right)$ incidence matrix $\widetilde{A}$ satisfies $\widetilde{A}^t\widetilde{A}=D+\alpha\alpha^t,$
where $\alpha=\left(\alpha_1,\ldots,\alpha_n\right)^t$ is a real vector with positive entries and $D$ is a diagonal matrix. Such a design is called \textit{uniform} if $D$ is a scalar matrix, i.e. $D= d I_{n}$ where $d$ is a real number and $I_n$ is the identity matrix of order $n$. According to Bridges and Mena \cite{Bridges4}, a graph whose adjacency matrix $A$ is the incidence matrix $\widetilde{A}$ of a uniform multiplicative design, where $\widetilde{A}$ is symmetric with trace zero, is called a \textit{multiplicative graph}. 

\begin{exemplo}\cite{Bridges}
\label{exemimp} 
The Fano  plane shown in Figure \ref{fanoplane} is a symmetric BIBD with parameters $\left(7,3,1\right)$. Let $\widetilde{X}$ be its $7\times 7$ incidence matrix and let $Y=J_7-\widetilde{X}$, where $J_7$ is the $7 \times 7$ matrix of all ones, and write $\overline{1}$ and $\overline{0}$ to represent the vectors of all ones and all zeros, respectively.

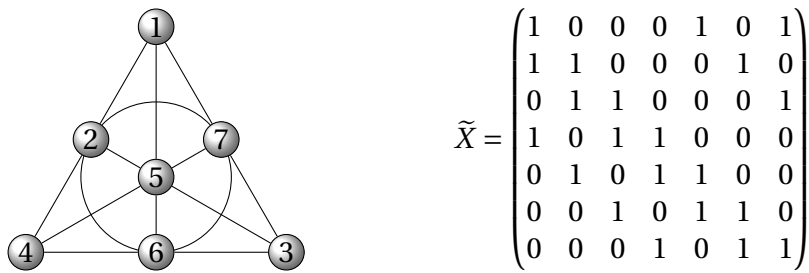
\begin{figure}[H]\centering
	\begin{minipage}{0.4\linewidth}\centering
	\[\begin{tikzpicture}[scale=1]
	\tikzstyle{point}=[ball color=white, circle, draw=black, inner sep=0.05cm]
	\node (v7) at (0,0) [point] {5};
	\draw (0,0) circle (1cm);
	\node (v1) at (90:2cm) [point] {1};
	\node (v2) at (210:2cm) [point] {4};
	\node (v4) at (330:2cm) [point] {3};
	\node (v3) at (150:1cm) [point] {2};
	\node (v6) at (270:1cm) [point] {6};
	\node (v5) at (30:1cm) [point] {7};
	\draw (v1) -- (v3) -- (v2);
	\draw (v2) -- (v6) -- (v4);
	\draw (v4) -- (v5) -- (v1);
	\draw (v3) -- (v7) -- (v4);
	\draw (v5) -- (v7) -- (v2);
	\draw (v6) -- (v7) -- (v1);
	\end{tikzpicture}\]
\end{minipage}
\begin{minipage}{0.4\linewidth}\centering
\[\widetilde{X} =\begin{pmatrix}
	1 & 0 & 0 & 0 & 1 & 0 & 1 \\ 
	1 & 1 & 0 & 0 & 0 & 1 & 0 \\
	0 & 1 & 1 & 0 & 0 & 0 & 1 \\
	1 & 0 & 1 & 1 & 0 & 0 & 0 \\
	0 & 1 & 0 & 1 & 1 & 0 & 0 \\
	0 & 0 & 1 & 0 & 1 & 1 & 0 \\
	0 & 0 & 0 & 1 & 0 & 1 & 1 
	\end{pmatrix}\]
\end{minipage}
\captionof{figure}{Fano plane and its incidence matrix}
\label{fanoplane}
\end{figure}

 The matrix $A$ given by
\begin{center}
	\[A = \left( \begin{array}{@{}c|c@{}}
	0_8 & 
	\begin{matrix}
	\overline{1}^{t} & \overline{0}^{t} \\ 
	\widetilde{X}&  Y\\ 
	\end{matrix}  \\
	\cmidrule[0.4pt]{1-2}
	\begin{matrix}
	\overline{1}  & \widetilde{X}^t \\ 
	\overline{0}  & Y^t\\ 
	\end{matrix}  & 
	\begin{matrix}
	J_7-I_7 & J_7-I_7 \\ 
	J_7-I_7 & J_7-I_7
	\end{matrix}  
	\end{array} \right) 
	\]
\end{center}
is the adjacency matrix of a multiplicative graph $G \in \mathcal{G}_n$ with $22$ vertices of degrees $7$ or $16$ and spectrum $\left\lbrace \left[14\right]^1,\left[2\right]^7,\left[-2\right]^{14}\right\rbrace$, shown in Figure \ref{grapag}.  
Here $A^2=4I+\alpha \alpha^t$ with $\alpha=\Big(\underbrace{\sqrt{3},\ldots,\sqrt{3}}_{8\text{-times}},\underbrace{2\sqrt{3},\ldots,2\sqrt{3}}_{14\text{-times}}\Big)^{t}$. This graph can also be described with  the points and planes of $AG\left(3,2\right)$, the three-dimensional affine space over the field $\mathbb{F}_2$. (See Example \ref{ex2novo})

\begin{figure}[H]\centering
\[\begin{tikzpicture}%[scale=0.5]
\node[
regular polygon,
regular polygon sides=22,
minimum size=4cm,
rotate=180/22,
] (a) {};
\draw[solid]
(a.corner 1) -- (a.corner 9)
(a.corner 1) -- (a.corner 10)
(a.corner 1) -- (a.corner 11)
(a.corner 1) -- (a.corner 12)
(a.corner 1) -- (a.corner 13)
(a.corner 1) -- (a.corner 14)
(a.corner 1) -- (a.corner 15)
(a.corner 2) -- (a.corner 9)
(a.corner 2) -- (a.corner 13)
(a.corner 2) -- (a.corner 15)
(a.corner 2) -- (a.corner 17)
(a.corner 2) -- (a.corner 18)
(a.corner 2) -- (a.corner 19)
(a.corner 2) -- (a.corner 21)
(a.corner 3) -- (a.corner 9)
(a.corner 3) -- (a.corner 10)
(a.corner 3) -- (a.corner 14)
(a.corner 3) -- (a.corner 18)
(a.corner 3) -- (a.corner 19)
(a.corner 3) -- (a.corner 20)
(a.corner 3) -- (a.corner 22)
(a.corner 4) -- (a.corner 10)
(a.corner 4) -- (a.corner 11)
(a.corner 4) -- (a.corner 15)
(a.corner 4) -- (a.corner 16)
(a.corner 4) -- (a.corner 19)
(a.corner 4) -- (a.corner 20)
(a.corner 4) -- (a.corner 21)
(a.corner 5) -- (a.corner 9)
(a.corner 5) -- (a.corner 11)
(a.corner 5) -- (a.corner 12)
(a.corner 5) -- (a.corner 17)
(a.corner 5) -- (a.corner 20)
(a.corner 5) -- (a.corner 21)
(a.corner 5) -- (a.corner 22)
(a.corner 6) -- (a.corner 10)
(a.corner 6) -- (a.corner 12)
(a.corner 6) -- (a.corner 13)
(a.corner 6) -- (a.corner 16)
(a.corner 6) -- (a.corner 18)
(a.corner 6) -- (a.corner 21)
(a.corner 6) -- (a.corner 22)
(a.corner 7) -- (a.corner 11)
(a.corner 7) -- (a.corner 13)
(a.corner 7) -- (a.corner 14)
(a.corner 7) -- (a.corner 16)
(a.corner 7) -- (a.corner 17)
(a.corner 7) -- (a.corner 19)
(a.corner 7) -- (a.corner 22)
(a.corner 8) -- (a.corner 12)
(a.corner 8) -- (a.corner 14)
(a.corner 8) -- (a.corner 15)
(a.corner 8) -- (a.corner 16)
(a.corner 8) -- (a.corner 17)
(a.corner 8) -- (a.corner 18)
(a.corner 8) -- (a.corner 20)
(a.corner 9) -- (a.corner 1)
(a.corner 9) -- (a.corner 2)
(a.corner 9) -- (a.corner 3)
(a.corner 9) -- (a.corner 5)
(a.corner 9) -- (a.corner 10)
(a.corner 9) -- (a.corner 11)
(a.corner 9) -- (a.corner 12)
(a.corner 9) -- (a.corner 13)
(a.corner 9) -- (a.corner 14)
(a.corner 9) -- (a.corner 15)
(a.corner 9) -- (a.corner 17)
(a.corner 9) -- (a.corner 18)
(a.corner 9) -- (a.corner 19)
(a.corner 9) -- (a.corner 20)
(a.corner 9) -- (a.corner 21)
(a.corner 9) -- (a.corner 22)
(a.corner 10) -- (a.corner 1)
(a.corner 10) -- (a.corner 3)
(a.corner 10) -- (a.corner 4)
(a.corner 10) -- (a.corner 6)
(a.corner 10) -- (a.corner 9)
(a.corner 10) -- (a.corner 11)
(a.corner 10) -- (a.corner 12)
(a.corner 10) -- (a.corner 13)
(a.corner 10) -- (a.corner 14)
(a.corner 10) -- (a.corner 15)
(a.corner 10) -- (a.corner 16)
(a.corner 10) -- (a.corner 18)
(a.corner 10) -- (a.corner 19)
(a.corner 10) -- (a.corner 20)
(a.corner 10) -- (a.corner 21)
(a.corner 10) -- (a.corner 22)
(a.corner 11) -- (a.corner 1)
(a.corner 11) -- (a.corner 4)
(a.corner 11) -- (a.corner 5)
(a.corner 11) -- (a.corner 7)
(a.corner 11) -- (a.corner 9)
(a.corner 11) -- (a.corner 10)
(a.corner 11) -- (a.corner 12)
(a.corner 11) -- (a.corner 13)
(a.corner 11) -- (a.corner 14)
(a.corner 11) -- (a.corner 15)
(a.corner 11) -- (a.corner 16)
(a.corner 11) -- (a.corner 17)
(a.corner 11) -- (a.corner 19)
(a.corner 11) -- (a.corner 20)
(a.corner 11) -- (a.corner 21)
(a.corner 11) -- (a.corner 22)
(a.corner 12) -- (a.corner 1)
(a.corner 12) -- (a.corner 5)
(a.corner 12) -- (a.corner 6)
(a.corner 12) -- (a.corner 8)
(a.corner 12) -- (a.corner 9)
(a.corner 12) -- (a.corner 10)
(a.corner 12) -- (a.corner 11)
(a.corner 12) -- (a.corner 13)
(a.corner 12) -- (a.corner 14)
(a.corner 12) -- (a.corner 15)
(a.corner 12) -- (a.corner 16)
(a.corner 12) -- (a.corner 17)
(a.corner 12) -- (a.corner 18)
(a.corner 12) -- (a.corner 20)
(a.corner 12) -- (a.corner 21)
(a.corner 12) -- (a.corner 22)
(a.corner 13) -- (a.corner 1)
(a.corner 13) -- (a.corner 2)
(a.corner 13) -- (a.corner 6)
(a.corner 13) -- (a.corner 7)
(a.corner 13) -- (a.corner 9)
(a.corner 13) -- (a.corner 10)
(a.corner 13) -- (a.corner 11)
(a.corner 13) -- (a.corner 12)
(a.corner 13) -- (a.corner 14)
(a.corner 13) -- (a.corner 15)
(a.corner 13) -- (a.corner 16)
(a.corner 13) -- (a.corner 17)
(a.corner 13) -- (a.corner 18)
(a.corner 13) -- (a.corner 19)
(a.corner 13) -- (a.corner 21)
(a.corner 13) -- (a.corner 22)
(a.corner 14) -- (a.corner 1)
(a.corner 14) -- (a.corner 3)
(a.corner 14) -- (a.corner 7)
(a.corner 14) -- (a.corner 8)
(a.corner 14) -- (a.corner 9)
(a.corner 14) -- (a.corner 10)
(a.corner 14) -- (a.corner 11)
(a.corner 14) -- (a.corner 12)
(a.corner 14) -- (a.corner 13)
(a.corner 14) -- (a.corner 15)
(a.corner 14) -- (a.corner 16)
(a.corner 14) -- (a.corner 17)
(a.corner 14) -- (a.corner 18)
(a.corner 14) -- (a.corner 19)
(a.corner 14) -- (a.corner 20)
(a.corner 14) -- (a.corner 22)
(a.corner 15) -- (a.corner 1)
(a.corner 15) -- (a.corner 2)
(a.corner 15) -- (a.corner 4)
(a.corner 15) -- (a.corner 8)
(a.corner 15) -- (a.corner 9)
(a.corner 15) -- (a.corner 10)
(a.corner 15) -- (a.corner 11)
(a.corner 15) -- (a.corner 12)
(a.corner 15) -- (a.corner 13)
(a.corner 15) -- (a.corner 14)
(a.corner 15) -- (a.corner 16)
(a.corner 15) -- (a.corner 17)
(a.corner 15) -- (a.corner 18)
(a.corner 15) -- (a.corner 19)
(a.corner 15) -- (a.corner 20)
(a.corner 15) -- (a.corner 21)
(a.corner 16) -- (a.corner 4)
(a.corner 16) -- (a.corner 6)
(a.corner 16) -- (a.corner 7)
(a.corner 16) -- (a.corner 8)
(a.corner 16) -- (a.corner 10)
(a.corner 16) -- (a.corner 11)
(a.corner 16) -- (a.corner 12)
(a.corner 16) -- (a.corner 13)
(a.corner 16) -- (a.corner 14)
(a.corner 16) -- (a.corner 15)
(a.corner 16) -- (a.corner 17)
(a.corner 16) -- (a.corner 18)
(a.corner 16) -- (a.corner 19)
(a.corner 16) -- (a.corner 20)
(a.corner 16) -- (a.corner 21)
(a.corner 16) -- (a.corner 22)
(a.corner 17) -- (a.corner 2)
(a.corner 17) -- (a.corner 5)
(a.corner 17) -- (a.corner 7)
(a.corner 17) -- (a.corner 8)
(a.corner 17) -- (a.corner 9)
(a.corner 17) -- (a.corner 11)
(a.corner 17) -- (a.corner 12)
(a.corner 17) -- (a.corner 13)
(a.corner 17) -- (a.corner 14)
(a.corner 17) -- (a.corner 15)
(a.corner 17) -- (a.corner 16)
(a.corner 17) -- (a.corner 18)
(a.corner 17) -- (a.corner 19)
(a.corner 17) -- (a.corner 20)
(a.corner 17) -- (a.corner 21)
(a.corner 17) -- (a.corner 22)
(a.corner 18) -- (a.corner 2)
(a.corner 18) -- (a.corner 3)
(a.corner 18) -- (a.corner 6)
(a.corner 18) -- (a.corner 8)
(a.corner 18) -- (a.corner 9)
(a.corner 18) -- (a.corner 10)
(a.corner 18) -- (a.corner 12)
(a.corner 18) -- (a.corner 13)
(a.corner 18) -- (a.corner 14)
(a.corner 18) -- (a.corner 15)
(a.corner 18) -- (a.corner 16)
(a.corner 18) -- (a.corner 17)
(a.corner 18) -- (a.corner 19)
(a.corner 18) -- (a.corner 20)
(a.corner 18) -- (a.corner 21)
(a.corner 18) -- (a.corner 22)
(a.corner 19) -- (a.corner 2)
(a.corner 19) -- (a.corner 3)
(a.corner 19) -- (a.corner 4)
(a.corner 19) -- (a.corner 7)
(a.corner 19) -- (a.corner 9)
(a.corner 19) -- (a.corner 10)
(a.corner 19) -- (a.corner 11)
(a.corner 19) -- (a.corner 13)
(a.corner 19) -- (a.corner 14)
(a.corner 19) -- (a.corner 15)
(a.corner 19) -- (a.corner 16)
(a.corner 19) -- (a.corner 17)
(a.corner 19) -- (a.corner 18)
(a.corner 19) -- (a.corner 20)
(a.corner 19) -- (a.corner 21)
(a.corner 19) -- (a.corner 22)
(a.corner 20) -- (a.corner 3)
(a.corner 20) -- (a.corner 4)
(a.corner 20) -- (a.corner 5)
(a.corner 20) -- (a.corner 8)
(a.corner 20) -- (a.corner 9)
(a.corner 20) -- (a.corner 10)
(a.corner 20) -- (a.corner 11)
(a.corner 20) -- (a.corner 12)
(a.corner 20) -- (a.corner 14)
(a.corner 20) -- (a.corner 15)
(a.corner 20) -- (a.corner 16)
(a.corner 20) -- (a.corner 17)
(a.corner 20) -- (a.corner 18)
(a.corner 20) -- (a.corner 19)
(a.corner 20) -- (a.corner 21)
(a.corner 20) -- (a.corner 22)
(a.corner 21) -- (a.corner 2)
(a.corner 21) -- (a.corner 4)
(a.corner 21) -- (a.corner 5)
(a.corner 21) -- (a.corner 6)
(a.corner 21) -- (a.corner 9)
(a.corner 21) -- (a.corner 10)
(a.corner 21) -- (a.corner 11)
(a.corner 21) -- (a.corner 12)
(a.corner 21) -- (a.corner 13)
(a.corner 21) -- (a.corner 15)
(a.corner 21) -- (a.corner 16)
(a.corner 21) -- (a.corner 17)
(a.corner 21) -- (a.corner 18)
(a.corner 21) -- (a.corner 19)
(a.corner 21) -- (a.corner 20)
(a.corner 21) -- (a.corner 22)
(a.corner 22) -- (a.corner 3)
(a.corner 22) -- (a.corner 5)
(a.corner 22) -- (a.corner 6)
(a.corner 22) -- (a.corner 7)
(a.corner 22) -- (a.corner 9)
(a.corner 22) -- (a.corner 10)
(a.corner 22) -- (a.corner 11)
(a.corner 22) -- (a.corner 12)
(a.corner 22) -- (a.corner 13)
(a.corner 22) -- (a.corner 14)
(a.corner 22) -- (a.corner 16)
(a.corner 22) -- (a.corner 17)
(a.corner 22) -- (a.corner 18)
(a.corner 22) -- (a.corner 19)
(a.corner 22) -- (a.corner 20)
(a.corner 22) -- (a.corner 21)
;
\fill[radius=1.5pt] \foreach \i in {1, ..., 22} { (a.corner \i) circle[] };
\end{tikzpicture}\]
\captionof{figure}{Graph on the points and planes of $AG\left(3,2\right)$}
\label{grapag}
\end{figure}
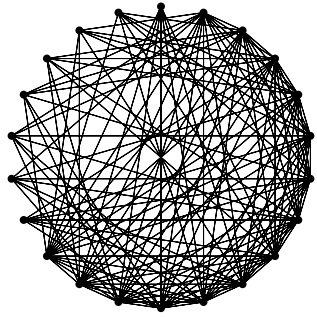
\end{exemplo}

It follows from Briges and Mena  \cite{Bridges4} that  a connected graph in $\mathcal{G}_n$ has three distinct nonzero eigenvalues if and only if it is a multiplicative graph. The result below shows that these graphs are integral with second largest eigenvalue greater than 1, and gives a constructive characterization if they are regular or irregular with second largest eigenvalue $\mu=2$. 

\begin{teo}\label{teoimp3}
A connected graph $G$ of order $n$ has spectrum  $\left\lbrace\left[\lambda\right]^1,\left[\mu\right]^{n-t-1},\left[-\mu\right]^t\right\rbrace$, where $1 \leq t \leq n-2$ and $\lambda>\mu >0$, if and only if $G$ is an integral multiplicative graph with second largest eigenvalue $\mu \geq 2$. Moreover,
\begin{itemize}
\item [(i)]  $G$ is regular if and only if it is a design graph with parameters $\left(n,\lambda,\lambda-\mu^2,\lambda-\mu^2\right)$.
\item [(ii)] $G$ is irregular with $\mu=2$ if and only if it is the cone over the Shrikhande graph,  the cone over the lattice graph $L_2\left(4\right)$ or the graph on the points and planes of $AG\left(3,2\right)$.
\item[(iii)] If $G$ is irregular with $\mu\geq 3$ then $G$ has more than $30$ vertices.
\end{itemize}
\end{teo}

\begin{proof}\renewcommand{\qedsymbol}{}
The fact that $G$ has spectrum of the form as stated above if and only if it is multiplicative is a result from  Bridges and Mena \cite[Theorem 4.1]{Bridges4}. In addition, since the spectrum of $G$ is not symmetric about zero, Lemma \ref{lemanovo} implies that $G$ is not bipartite. It follows from Lemma \ref{teo8} that $\lambda\in\mathbb{Z}$. On the other hand, since $\lambda+\left(n-t-1\right)\mu-t\mu=0$, we have $\lambda=\left(2t-n+1\right)\mu$. Hence $\mu\in\mathbb{Q}$, and the rational root theorem implies that $\mu\in\mathbb{Z}$. Then $G$ is integral. 
Beside that, since $G$ has three distinct eigenvalues, its diameter is two and so $G$ contains the path on three vertices $P_3$ as an induced subgraph. Thus, by Cauchy's interlacing theorem, $-\mu\leq-\sqrt{2}$. Therefore $\mu \geq 2$.

Suppose that $G$ is regular. Hence $G$ is a  $srg\left(n,\lambda,\alpha,\beta\right)$ with $\mu = \frac{\left(\alpha-\beta\right)+\sqrt{\left(\alpha-\beta\right)^2+4\left(\lambda-\beta\right)}}{2}$ and $-\mu = \frac{\left(\alpha-\beta\right)-\sqrt{\left(\alpha-\beta\right)^2+4\left(\lambda-\beta\right)}}{2}$, by Lemmas \ref{treesrg} and \ref{lemasrg}. Therefore $\alpha=\beta$ and so $\mu =\sqrt{\lambda-\beta}$. Thus $G$ is a design graph with parameters $\left(n,\lambda,\lambda-\mu^2,\lambda-\mu^2\right)$. The converse is true since every design graph is a regular graph. 
Now suppose that $G$ is irregular. Van Dam \cite[Theorem 7 and Table II]{VANDAM} characterized  all connected  graphs with three distinct eigenvalues, which are not strongly regular or complete bipartite, that have at most 30 vertices or with each eigenvalue at least $-2$. Statements $(ii)$ and $(iii)$ follow by inspecting the spectrum of each graph in his results. \QEDA 

\end{proof}

Although the characterization given in Theorem \ref{teoimp3} is not constructive in case $(iii)$, it seems that to obtain such a  characterization for  all irregular graphs in $\mathcal{G}_n$ with three nonzero distinct eigenvalues is hard to accomplish, since we can obtain infinite families of irregular connected graphs in $\mathcal{G}_n$ with spectrum of this form. Below we present two such families.

\begin{exemplo}\label{ex1}

Ahrens and Szekeres \cite{Ahrens} showed that there exist strongly regular graphs with parameters $\left(\alpha^3 + 2\alpha^2,\alpha^2+\alpha,\alpha,\alpha\right)$ for all prime power values of $\alpha$. On the other hand, it follows from Abreu et al. \cite[Proposition 3]{AbreuN} that a cone $C_{\alpha}$ over a connected $\left(\alpha^2+\alpha\right)$-regular multiplicative graph $G$ with three distinct eigenvalues $\alpha^2+\alpha>\alpha>-\alpha$ is multiplicative if and only if $G$ is strongly regular with parameters $\left(\alpha^3+2\alpha^2,\alpha^2+\alpha,\alpha,\alpha\right)$. 
In this case, by \cite[Lemma 4.1]{MUZY}
\begin{equation}\label{coneeq}
Spec\left(C_{\alpha}\right)=\left\lbrace\left[\alpha^2+2\alpha\right]^1,\left[\alpha\right]^{\frac{\alpha^3+2\alpha^2-\alpha-2}{2}},\left[-\alpha\right]^ {\frac{\alpha^3+2\alpha^2+\alpha+2}{2}}\right\rbrace.
\end{equation}

Hence for every prime power $\alpha$ a cone $C_{\alpha}$ over $G=srg\left(\alpha^3 + 2\alpha^2,\alpha^2+\alpha,\alpha,\alpha\right)$ is an irregular graph in $\mathcal{G}_n$ with three nonzero distinct eigenvalues. The examples with smallest order in this infinite family, obtained using $\alpha=2$, are the cone over the Shrikhande graph and the cone over the lattice graph $L_2\left(4\right)$, which have $17$ vertices. Note that with $\alpha\geq 3$ we obtain cones with at least $46$ vertices.
\end{exemplo}

\begin{exemplo}\label{ex2novo}
Van Dam \cite{VANDAM} introduced another family with infinitely many irregular graphs in $\mathcal{G}_n$ with three nonzero distinct eigenvalues, which are not cones.  The \textit{incidence graph} of a BIBD with $n$ points and $b$ blocks is the bipartite graph of order $n+b$ with two vertices adjacent if and only if one corresponds to a block and the other corresponds to an element contained in that block. Each graph in the family presented in \cite{VANDAM} is \mbox{constructed} from the incidence graph of a BIBD with parameters $\left(q^3,q^2,q+1\right)$ on the points and planes of $AG\left(3,q\right)$, the three-dimensional affine space over $\mathbb{F}_q$, by  adding an edge between two blocks if they intersect (in $q$ points). 
The construction yields an irregular graph with three distinct eigenvalues with spectrum $$\left\lbrace\left[q^3+q^2+q\right]^1,\left[q\right]^{q^3-1},\left[-q\right]^{q^3+q^2+q}\right\rbrace.$$

The smallest example in this infinite family, obtained with $q=2$, is the graph depicted in Figure \ref{grapag}. Note that if $q\geq 3$ we obtain graphs with at least $66$ vertices.
\end{exemplo}
 
%%%%%%%%%%%%%%%%%%%%%%%%%%%%%%%%%%%%%%%%%%%%%%%%%%%%%%%%%%%%%%%%%%%%%%%%%

\section{Four distinct eigenvalues}\label{sec:III-I}

In this section we study connected graphs in $\mathcal{G}_n$ that have exactly four distinct eigenvalues, i.e. with spectrum $\left\lbrace\left[\lambda\right]^1,\left[\mu\right]^{n-k-t-1},\left[0\right]^{t},\left[-\mu\right]^k\right\rbrace$, where $t,k \geq 1$, $t +k \leq n-2$, and $\lambda>\mu>0.$  We characterize these graphs in the case their are regular with at least two simple eigenvalues and present infinite families of graphs in the other cases.

The incidence graphs of symmetric balanced incomplete block designs are examples of regular graphs with four distinct eigenvalues. 
In fact these are the only bipartite graphs with four distinct eigenvalues.

\begin{lema}\cite{Cvetkovic}\label{lema2}
A connected bipartite regular graph $G$ with four distinct eigenvalues is the incidence graph of a symmetric \textit{BIBD} with parameters $\left(n,r,\alpha\right)$. The spectrum of $G$ is given by $$\left\lbrace\left[r\right]^1,\left[\sqrt{r-\alpha}\right]^{n-1},
\left[-\sqrt{r-\alpha}\right]^{n-1},\left[-r\right]^{1}\right\rbrace.$$\end{lema}

\begin{exemplo}\cite{VanDam1995}\label{ex34}
The graph obtained by removing a perfect matching from the complete bipartite graph $K_{\ell,\ell}$, denoted by $K^{-}_{\ell, \ell}$, is the incidence graph of a symmetric \textit{BIBD} with parameters $\left(\ell,\ell-1,\ell -2\right)$, for $\ell>2$. By Lemma \ref{lema2}, $$Spec\left(K^{-}_{\ell, \ell} \right) = \left\lbrace\left[\ell -1 \right]^1,\left[1\right]^{\ell-1},\left[-1\right]^{\ell -1},\left[-\ell + 1 \right]^1 \right\rbrace.$$
We note that  $K_{4,4}^{-}$ is the cubical graph $Q_3$ formed by the 8 vertices and 12 edges of a three-dimensional cube.
\end{exemplo}

Another family of connected regular graphs with four distinct eigenvalues, uniquely determined by their spectrum,  was given by Van Dam \cite{VanDam1995}. The graphs in the family are obtained  by a product construction with the graph $K^{-}_{\ell, \ell}$. Recall that the \textit{Kronecker product} $R\otimes S$ of the matrices $R=\left(r_{ij}\right)_{c\times d}$ and $S=\left(s_{ij}\right)_{p\times q}$ is the $cp\times dq$ matrix obtained from $R$ by replacing each element $
r_{ij}$ with the block $r_{ij}S$. Given a graph $G$ of order $n$ with adjacency matrix $A$, we denote by  $G\otimes J_m$ the graph with adjacency matrix $A\otimes J_m$, and by  $G\circledast J_m$ the graph with adjacency matrix $A\circledast J_m=\left(A+I_n\right)\otimes J_m-I_{nm}$, where $I_n$ denotes the identity matrix of size $n$ and $J_m$ represents the $m \times m$ matrix of all ones. Note that $G \otimes J_1 = G = G\circledast J_1$ and $\overline{G\otimes J_m} = \overline{G} \circledast J_m$, where $\overline{G}$  is the complement of $G$. In addition, if $G$ is connected and regular, then $G\otimes J_m$ and $G\circledast J_m$ are connected and regular. 

\begin{lema}\label{lema31}
Let $G$ be a graph of order $n$ with index $\lambda_1$, spectrum $\left\lbrace\left[\lambda_1\right]^{m_1}, \left[\lambda_2\right]^{m_2}, \ldots, \left[\lambda_t\right]^{m_t}\right\rbrace$, and complement $\overline{G}$. Then
\begin{enumerate}
\item[$(i)$] If $G$ is regular then  $\overline{G}$
is regular with spectrum $$\left\lbrace\left[n-1-\lambda_1\right]^{m_1}, \left[-\lambda_{2}-1\right]^{m_2}, \ldots, \left[-\lambda_{t}-1\right]^{m_t}\right\rbrace.$$

\item[$(ii)$] $G$ is regular if and only if $G\otimes J_m$ is regular, for all $m \geq 1$. Moreover,
 $$Spec\left(G\otimes J_m\right)= \left\lbrace\left[m\lambda_1\right]^{m_1}, \ldots, \left[m\lambda_t\right]^{m_t}, \left[0\right]^{n\left(m-1\right)}\right\rbrace.$$

\item[$(iii)$] $G$ is regular if and only if $G\circledast J_m$ is regular, for all $m \geq 1$. Moreover,
 $$Spec\left(G\circledast J_m\right)= \left\lbrace\left[m\lambda_1+m-1\right]^{m_1}, \ldots, \left[m\lambda_t+m-1\right]^{m_t}, \left[-1\right]^{n\left(m-1\right)}\right\rbrace.$$
\end{enumerate}
\end{lema}

\begin{proof}\renewcommand{\qedsymbol}{}
A proof of $(i)$ is given in \cite[Theorem 6.15]{Bapat}. 
%The eigenvalues of $J_m$ are $m$ together with $m-1$ zeros % \cite[Example 1.3.23]{Horn}. 
 It is shown in \cite[Lemma 3.25]{Bapat} that for any symmetric matrices $A$ and $B$ the eigenvalues of $A \otimes B$ are given by $\beta_i\gamma_j$, $1\leq i \leq n$ and \mbox{$1 \leq j \leq m$,} where $\beta_1, \dots, \beta_n$ and $\gamma_1, \dots, \gamma_m$ are the eigenvalues of $A$ and $B$, respectively. Hence, since $Spec\left(J_m\right) = \left\{\left[m\right]^1, \left[0\right]^{m-1}\right\}$, we get $$Spec\left(G\otimes J_m\right)= \left\lbrace\left[m\lambda_1\right]^{m_1}, \ldots, \left[m\lambda_t\right]^{m_t}, \left[0\right]^{n\left(m-1\right)}\right\rbrace.$$ 
By \cite[Theorem 3.22]{Cvetkovic}, $G$ is regular if and only if $\frac{1}{n} \sum_{i=1}^{t} m_i \lambda_{i}^{2} = \lambda_1 $. In addition, note that since $m \geq 1$ we have
$$ \frac{1}{n} \sum_{i=1}^{t} m_i \lambda_{i}^{2} = \lambda_1 \Longleftrightarrow \frac{1}{nm} \sum_{i=1}^{t} m_i (m\lambda_{i})^{2} = m\lambda_1, $$
which concludes the proof of $(ii)$. The proof of $(iii)$ is similar. \QEDA

\end{proof}

\begin{exemplo}
\label{ex1secao3}
It was proved in \cite{VanDam1995} that $K_{\ell, \ell}^{-}\circledast J_m$ is uniquely determined by its spectrum, for each $\ell$ and $m$. Applying Lemma \ref{lema31} we can easily see that  $K_{\ell, \ell}^{-}\circledast J_m$ has four distinct eigenvalues. In particular, taking $\ell=4$ we obtain an infinite family of connected regular graphs with four distinct eigenvalues, one equals $-1$: $$Spec\left(Q_3\circledast J_m\right)= \left\lbrace\left[4m-1\right]^{1},\left[2m-1\right]^{3},\left[-1\right]^{8m-5}, \left[-2m-1\right]^{1}\right\rbrace,$$ for all $m \geq 1$. This family generates another infinite family of connected regular graphs with four distinct eigenvalues, all belonging to the family $\mathcal{G}_n$. In fact Lemma \ref{lema31} implies that  $\overline{Q_3\circledast J_m}$ is a regular graph with $$Spec\left(\overline{Q_3\circledast J_m}\right)= \left\lbrace\left[4m\right]^{1},\left[2m\right]^{1}, \left[0\right]^{8m-5}, \left[-2m\right]^{3}\right\rbrace,$$ so it is a connected  integral graph in $\mathcal{G}_{n}$, with $n=8m$. Figure \ref{graphKJ} shows the  graphs obtained with $m=1$ and $m=2$.
\end{exemplo}

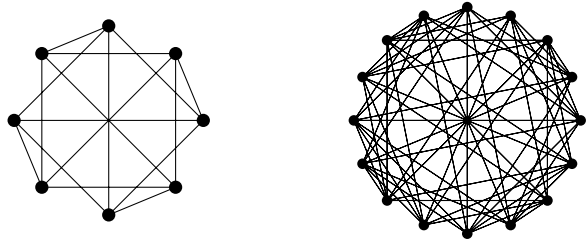
\begin{figure}[H]\centering
\begin{minipage}{0.3\linewidth}\centering
	\[\begin{tikzpicture}[scale=1]
	\node[
	regular polygon,
	regular polygon sides=8,
	minimum size=2.5cm,
	rotate=180/8,
	] (a) {};
	\draw[solid]
	(a.corner 1) -- (a.corner 2)
	(a.corner 1) -- (a.corner 3)
	(a.corner 1) -- (a.corner 5)
	(a.corner 1) -- (a.corner 7)
	(a.corner 2) -- (a.corner 4)
	(a.corner 2) -- (a.corner 6)
	(a.corner 2) -- (a.corner 8)
	(a.corner 3) -- (a.corner 4)
	(a.corner 3) -- (a.corner 5)
	(a.corner 3) -- (a.corner 7)
	(a.corner 4) -- (a.corner 6)
	(a.corner 4) -- (a.corner 8)
	(a.corner 5) -- (a.corner 6)
	(a.corner 5) -- (a.corner 7)
	(a.corner 6) -- (a.corner 8)
	(a.corner 7) -- (a.corner 8)
	;
	\fill[radius=2.5pt] \foreach \i in {1, ..., 8} { (a.corner \i) circle[] };
	\end{tikzpicture}\]
\end{minipage}
\begin{minipage}{0.3\linewidth}\centering
	\[\begin{tikzpicture}[scale=1]
	\node[
	regular polygon,
	regular polygon sides=16,
	minimum size=3cm,
	rotate=180/16,
	] (a) {};
	\draw[solid]
	(a.corner 1) -- (a.corner 3)
	(a.corner 1) -- (a.corner 4)
	(a.corner 1) -- (a.corner 5)
	(a.corner 1) -- (a.corner 6)
	(a.corner 1) -- (a.corner 9)
	(a.corner 1) -- (a.corner 10)
	(a.corner 1) -- (a.corner 13)
	(a.corner 1) -- (a.corner 14)
	(a.corner 2) -- (a.corner 3)
	(a.corner 2) -- (a.corner 4)
	(a.corner 2) -- (a.corner 5)
	(a.corner 2) -- (a.corner 6)
	(a.corner 2) -- (a.corner 9)
	(a.corner 2) -- (a.corner 10)
	(a.corner 2) -- (a.corner 13)
	(a.corner 2) -- (a.corner 14)
	(a.corner 3) -- (a.corner 1)
	(a.corner 3) -- (a.corner 2)
	(a.corner 3) -- (a.corner 7)
	(a.corner 3) -- (a.corner 8)
	(a.corner 3) -- (a.corner 11)
	(a.corner 3) -- (a.corner 12)
	(a.corner 3) -- (a.corner 15)
	(a.corner 3) -- (a.corner 16)
	(a.corner 4) -- (a.corner 1)
	(a.corner 4) -- (a.corner 2)
	(a.corner 4) -- (a.corner 7)
	(a.corner 4) -- (a.corner 8)
	(a.corner 4) -- (a.corner 11)
	(a.corner 4) -- (a.corner 12)
	(a.corner 4) -- (a.corner 15)
	(a.corner 4) -- (a.corner 16)
	(a.corner 5) -- (a.corner 1)
	(a.corner 5) -- (a.corner 2)
	(a.corner 5) -- (a.corner 7)
	(a.corner 5) -- (a.corner 8)
	(a.corner 5) -- (a.corner 9)
	(a.corner 5) -- (a.corner 10)
	(a.corner 5) -- (a.corner 13)
	(a.corner 5) -- (a.corner 14)
	(a.corner 6) -- (a.corner 1)
	(a.corner 6) -- (a.corner 2)
	(a.corner 6) -- (a.corner 7)
	(a.corner 6) -- (a.corner 8)
	(a.corner 6) -- (a.corner 9)
	(a.corner 6) -- (a.corner 10)
	(a.corner 6) -- (a.corner 13)
	(a.corner 6) -- (a.corner 14)
	(a.corner 7) -- (a.corner 3)
	(a.corner 7) -- (a.corner 4)
	(a.corner 7) -- (a.corner 5)
	(a.corner 7) -- (a.corner 6)
	(a.corner 7) -- (a.corner 11)
	(a.corner 7) -- (a.corner 12)
	(a.corner 7) -- (a.corner 15)
	(a.corner 7) -- (a.corner 16)
	(a.corner 8) -- (a.corner 3)
	(a.corner 8) -- (a.corner 4)
	(a.corner 8) -- (a.corner 5)
	(a.corner 8) -- (a.corner 6)
	(a.corner 8) -- (a.corner 11)
	(a.corner 8) -- (a.corner 12)
	(a.corner 8) -- (a.corner 15)
	(a.corner 8) -- (a.corner 16)
	(a.corner 9) -- (a.corner 1)
	(a.corner 9) -- (a.corner 2)
	(a.corner 9) -- (a.corner 5)
	(a.corner 9) -- (a.corner 6)
	(a.corner 9) -- (a.corner 11)
	(a.corner 9) -- (a.corner 12)
	(a.corner 9) -- (a.corner 13)
	(a.corner 9) -- (a.corner 14)
	(a.corner 10) -- (a.corner 1)
	(a.corner 10) -- (a.corner 2)
	(a.corner 10) -- (a.corner 5)
	(a.corner 10) -- (a.corner 6)
	(a.corner 10) -- (a.corner 11)
	(a.corner 10) -- (a.corner 12)
	(a.corner 10) -- (a.corner 13)
	(a.corner 10) -- (a.corner 14)
	(a.corner 11) -- (a.corner 3)
	(a.corner 11) -- (a.corner 4)
	(a.corner 11) -- (a.corner 7)
	(a.corner 11) -- (a.corner 8)
	(a.corner 11) -- (a.corner 9)
	(a.corner 11) -- (a.corner 10)
	(a.corner 11) -- (a.corner 15)
	(a.corner 11) -- (a.corner 16)
	(a.corner 12) -- (a.corner 3)
	(a.corner 12) -- (a.corner 4)
	(a.corner 12) -- (a.corner 7)
	(a.corner 12) -- (a.corner 8)
	(a.corner 12) -- (a.corner 9)
	(a.corner 12) -- (a.corner 10)
	(a.corner 12) -- (a.corner 15)
	(a.corner 12) -- (a.corner 16)
	(a.corner 13) -- (a.corner 1)
	(a.corner 13) -- (a.corner 2)
	(a.corner 13) -- (a.corner 5)
	(a.corner 13) -- (a.corner 6)
	(a.corner 13) -- (a.corner 9)
	(a.corner 13) -- (a.corner 10)
	(a.corner 13) -- (a.corner 15)
	(a.corner 13) -- (a.corner 16)
	(a.corner 14) -- (a.corner 1)
	(a.corner 14) -- (a.corner 2)
	(a.corner 14) -- (a.corner 5)
	(a.corner 14) -- (a.corner 6)
	(a.corner 14) -- (a.corner 9)
	(a.corner 14) -- (a.corner 10)
	(a.corner 14) -- (a.corner 15)
	(a.corner 14) -- (a.corner 16)
	(a.corner 15) -- (a.corner 3)
	(a.corner 15) -- (a.corner 4)
	(a.corner 15) -- (a.corner 7)
	(a.corner 15) -- (a.corner 8)
	(a.corner 15) -- (a.corner 11)
	(a.corner 15) -- (a.corner 12)
	(a.corner 15) -- (a.corner 13)
	(a.corner 15) -- (a.corner 14)
	(a.corner 16) -- (a.corner 3)
	(a.corner 16) -- (a.corner 4)
	(a.corner 16) -- (a.corner 7)
	(a.corner 16) -- (a.corner 8)
	(a.corner 16) -- (a.corner 11)
	(a.corner 16) -- (a.corner 12)
	(a.corner 16) -- (a.corner 13)
	(a.corner 16) -- (a.corner 14)
	;
	\fill[radius=2pt] \foreach \i in {1, ..., 16} { (a.corner \i) circle[] };
	\end{tikzpicture}\]
\end{minipage} 
	\captionof{figure}{$\overline{Q_3}$ and $\overline{Q_3\circledast J_{2}}$ graphs}
	\label{graphKJ}
\end{figure}

\begin{lema}\label{gbar}
Let $G$ be a  connected $r$-regular graph of order $n$ and four distinct eigenvalues. Then its complement $\overline{G}$ is also connected and regular with four distinct eigenvalues, or $\overline{G}$  is disconnected, and then it is the union of cospectral strongly regular graphs. 
\end{lema}

\begin{proof}\renewcommand{\qedsymbol}{}
Let $Spec\left(G\right)=\left\{[r]^{1}, [\lambda_2]^{m_2}, [\lambda_3]^{m_3}, [\lambda_4]^{m_4}\right\},$ 
where $r>\lambda_2>\lambda_3>\lambda_4$. Then, by Lemma \ref{lema31}, $\overline{G}$ is $(n-1-r)$-regular with spectrum $$\left\{[n-1-r]^1, [-\lambda_2-1]^{m_2} , [-\lambda_3 -1]^{m_3}, [ -\lambda_4-1]^{m_4}\right\}.$$ 

Either the index of $\overline{G}$ is simple and so it is connected  with four distinct eigenvalues, or the index of $\overline{G}$ is not simple and then it is disconnected  with  three distinct eigenvalues. The spectrum of a disconnected graph is the union of the spectra of its connected components. Hence in the case $\overline{G}$ is disconnected, each connected component has two or three distinct eigenvalues.  Lemma \ref{treesrg} implies that every component with three distinct eigenvalues is a strongly regular graph, since  $\overline{G}$ is regular. If a connected component has only two distinct eigenvalues, then it is a complete graph, which has least eigenvalue $-1$. As shown in the proof of Theorem \ref{teoimp3} the smallest eigenvalue of a connected graph with 3 distinct eigenvalues is at most $ -\sqrt{2}$. Hence in the case $\overline{G}$ has a complete graph as a connected component, each component with three distinct eigenvalues has at least two negative eigenvalues, namely $-1$ and $-\lambda_2 - 1 < -1$. However, by Lemma \ref{lemasrg} every  strongly regular graph has only one negative eigenvalue. Therefore, if $\overline{G}$ is disconnect, all its connected components are cospectral strongly regular graphs. \QEDA

\end{proof}

\begin{lema}\cite[Theorem 3.5]{Huang}\label{teo13}
There are no connected $r$-regular graphs with spectrum $\left\lbrace\left[r\right]^1,\left[-1\right]^{1},\left[\delta\right]^{m},\left[\zeta\right]^{n-2-m}\right\rbrace$, where $\delta$ and $\zeta$ are integers and $2\leq m\leq n-4$.
\end{lema}

In \cite{Huang} Huang and Huang also characterized all connected regular graphs with four distinct eigenvalues, one equals $0$, where exactly two are simple.

\begin{lema}\cite[Theorem 3.7]{Huang}
\label{huanghuang}
A connected regular graph $G$ has four distinct eigenvalues in which exactly two eigenvalues are simple and with $0$ as an eigenvalue if and only if $G=\overline{K^{-}_{\ell,\ell}\circledast J_m}$, with $\ell\geq 3$ and $m\geq 1$. 
\end{lema}

The result below characterizes all connected regular graphs in $\mathcal{G}_n$ with four distinct eigenvalues, at least two of them simple.  In particular it shows that the graphs in $\mathcal{G}_n$ in the family given in Example  \ref{ex1secao3} are the only connected graphs in $\mathcal{G}_n$ with four distinct eigenvalues where the index is not the only simple eigenvalue.

\begin{teo}\label{teoimp4}
Let $G$ be a connected graph of order $n$ with spectrum $\left\lbrace\left[\lambda\right]^1,\left[\mu\right]^{n-k-t-1},\left[0\right]^{t},\right.$ $\left.\left[-\mu\right]^{k}\right\rbrace$, where  $t,k \geq 1$, $t +k \leq n-2$, and $\lambda>\mu>0$. Then $G$ is integral. Furthermore, if $G$ is regular then either $G=\overline{Q_3\circledast J_{\frac{\mu}{2}}}$ or $\lambda$ is the only simple eigenvalue of $G$.
\end{teo} 

\begin{proof}\renewcommand{\qedsymbol}{}
The adjacency matrix of a graph is diagonalizable and so its minimal polynomial is a product of distinct linear factors (see for instance \cite{Horn}). Thus the minimum polynomial of $G$ is
\begin{center}
$m_G\left(x\right)=x\left(x-\lambda\right)\left(x-\mu\right)\left(x+\mu\right)=x^4-\lambda x^3-\mu^2x^2+\lambda\mu^2 x.$
\end{center}

Beside that, $m_G\left(x\right)$ has integral coefficients \cite[Lemma 2.5]{VanDam1995} which implies that $\lambda\in\mathbb{Z}$. Since  $\lambda+\left(n-k-t-1\right)\mu-k\mu=0$, we have  $\lambda=\left(2k+t-n+1\right)\mu$. Hence  $\mu\in\mathbb{Q}$ and the rational root theorem implies that $\mu\in\mathbb{Z}$. Then $G$ is integral.
	
Suppose that $G$ is regular. We first prove that $0$ is not a simple eigenvalue of $G$. In fact if $t=1$  Lemma \ref{lema31} implies that $\overline{G}$ is regular with spectrum 
$$Spec\left(\overline{G}\right) =\left\lbrace\left[n-1-\lambda\right]^1,\left[\mu-1\right]^{k},\left[-1\right]^{1},\left[-\mu-1\right]^{n-k-2}\right\rbrace.$$

Clearly $-1$ cannot be equal to any other eigenvalue of $\overline{G}$, thus it is a simple eigenvalue. On the other hand, by Lemma  \ref{gbar} $\overline{G}$ is connected with four distinct eigenvalues or a disjoint union of cospectral strongly regular graphs.  The former case is excluded by  Lemma \ref{teo13}. The latter case also cannot occur since  $-1$ is  simple. Therefore $0$ is not a simple eigenvalue of $G$.
	
It is easy to see that $-\mu$ is also not a simple eigenvalue of $G$. If $k=1$ we have $\lambda=\left(t-n+3\right)\mu$ and so $t-n+3\geq 2$. Hence $n\leq t+1$, which  contradicts the fact that $n \geq t+3$. 
  
Now suppose that $\mu$ is a simple eigenvalue of $G$. Then $G$ has exactly two simple eigenvalues and Lemma \ref{huanghuang} implies that $\overline{G} =K^{-}_{\ell, \ell}\circledast J_m$, with $\ell\geq 3$, $m\geq 1$ and $n=2\ell m$. By lemma \ref{lema31} we have $$Spec\left(\overline{G}\right)=\left\lbrace\left[\ell m -1 \right]^1,\left[(2-\ell)m-1\right]^{1},\left[-1\right]^{2\ell m-\ell-1},\left[2m-1\right]^{\ell-1 }\right\rbrace,$$ and hence $$Spec\left(G\right)=\left\lbrace\left[\ell m \right]^1, \left[(\ell -2)m \right]^{1},\left[0\right]^{2\ell m-\ell-1},\left[-2m\right]^{\ell-1}\right\rbrace.$$

Therefore $\lambda=\ell m$, $\mu=\left(\ell -2\right)m$ and $-\mu =-2m$, which implies that $\ell=4$. Thus \mbox{$G=\overline{Q_3\circledast J_{\frac{\mu}{2}}}$.} \QEDA 
	
\end{proof}

For a complete characterization of connected graphs in $\mathcal{G}_n$ with four distinct  eigenvalues it remains to consider irregular graphs and regular graphs where the index is the only simple eigenvalue. For both cases we present next infinite families of graphs. Our constructions are based on the result below, which follows from Lema \ref{lema31}.

\begin{prop}\label{prop3}
Let $G$ be a connected graph of order $n$. If $G \in \mathcal{G}_n,$ then $G \otimes J_m \in \mathcal{G}_{nm}$ for all $m \geq 1$. 
%Moreover, $G$ is regular if and only if $G \otimes J_m $ is regular.
%if $Spec\left(G\right) = \left\lbrace\left[\lambda\right]^1,\left[\mu\right]^{n-k-t-1},\left[0\right]^{t},\right.$ $\left.\left[-\mu\right]^{k}\right\rbrace$, where  $t \geq 0$, $k \geq 1$,  $t+k \leq n-2$, and $\lambda>\mu>0$,  then 
%$$Spec\left(G\otimes J_m\right) = \left\lbrace\left[m\lambda\right]^1,\left[m\mu\right]^{n-k-t-1},\left[0\right]^{t+ (m-1)n},\right. \left.\left[-m\mu\right]^{k}\right\rbrace$$
\end{prop}

\begin{exemplo}\label{fam1irr4}
It follows from Example \ref{ex1} that for every prime power $\alpha$ a cone $C_{\alpha}$ over \mbox{$G=srg\left(\alpha^3 + 2\alpha^2,\alpha^2+\alpha,\alpha,\alpha\right)$} is an irregular graph in $\mathcal{G}_n$ with spectrum given by (\ref{coneeq}).
Therefore,  $C_{\alpha} \otimes J_m$ is an irregular graph in $\mathcal{G}_{nm}$ with four distinct eigenvalues such that
$$Spec\left(C_{\alpha}\otimes J_m \right)=\left\lbrace\left[m (\alpha^2+2\alpha)\right]^1,\left[m\alpha\right]^{\frac{\alpha^3+2\alpha^2-\alpha-2}{2}},\left[0\right]^{n(m-1)}, \left[-m\alpha\right]^ {\frac{\alpha^3+2\alpha^2+\alpha+2}{2}}\right\rbrace,$$
for all $m \geq 2$.
The graphs of smallest order in this family (34 vertices), obtained with $\alpha=2$, are $C_2^1 \otimes J_2$ and $C_2^2 \otimes J_2$, where $C_2^1$ is the cone over the Shrikhande graph and $C_2^2$ is  the cone over the lattice graph $L_2\left(4\right)$. Their spectrum is $\left\lbrace \left[16\right]^1,\left[4\right]^6, \left[0\right]^{17}, \left[-4\right]^{10}\right\rbrace$.
\end{exemplo}

\begin{exemplo}\label{fam2irr4}
In Example \ref{ex2novo} we presented an infinite family of irregular graphs in $\mathcal{G}_n$ with spectrum given by 
$\left\lbrace\left[q^3+q^2+q\right]^1,\left[q\right]^{q^3-1},\left[-q\right]^{q^3+q^2+q}\right\rbrace$, for every prime power $q$. 
Taking a graph $G$ of order $n$ in this family and an integer $m\geq 2$, the graph $G \otimes J_m$ is a irregular graph in $\mathcal{G}_{nm}$ with four distinct eigenvalues such that $$Spec\left(G\otimes J_m \right)=\left\lbrace\left[m(q^3+q^2+q)\right]^1,\left[mq\right]^{q^3-1}, \left[0\right]^{n(m-1)}, \left[-mq\right]^{q^3+q^2+q}\right\rbrace.$$The graph of smallest order in this infinite family is  $G \otimes J_2$, where $G$ is the graph on the points and planes of $AG(3, 2)$ (Figure \ref{grapag}).  $G \otimes J_2$ is an irregular graph on $44$ vertices and spectrum  $\left\lbrace \left[28\right]^1,\left[4\right]^7, \left[0\right]^{22}, \left[-4\right]^{14}\right\rbrace$.
\end{exemplo}

\begin{exemplo}
\label{reg4srg}
From each design graph in $\mathcal{G}_n$  we can construct an infinite family of connected regular graphs with four distinct eigenvalues. It follows from Theorem \ref{teoimp3} and Lemma \ref{lemasrg} that a regular graph $G$ in $\mathcal{G}_n$ with three distinct nonzero eigenvalues 
is a design graph with parameters $\left(n,\lambda,\lambda-\mu^2,\lambda-\mu^2\right)$ and spectrum  $\left\lbrace\left[\lambda\right]^1,\left[\mu\right]^{\frac{1}{2}\left(n - 1 -\frac{\lambda}{\mu} \right)},\left[-\mu\right]^{\frac{1}{2}\left(n - 1 +\frac{\lambda}{\mu} \right)}\right\rbrace$. Thus $G \otimes J_m$ is a regular graph in $\mathcal{G}_{nm}$ with four distinct eigenvalues such that
$$Spec\left(G\otimes J_m \right)=\left\lbrace\left[m\lambda\right]^1,\left[m\mu\right]^{\frac{1}{2}\left(n - 1 -\frac{\lambda}{\mu} \right)}, \left[0\right]^{n(m-1)}, \left[-m\mu\right]^{\frac{1}{2}\left(n - 1 +\frac{\lambda}{\mu} \right)}\right\rbrace$$
for all $m \geq 2$. Note that the index is the only simple eigenvalue of $G\otimes J_m$. The graph of smallest order in this family is obtained with $m=2$ and  $G$  a design graph with parameters $\left(15, 8, 4, 4\right)$. Such design is unique and isomorphic to the line graph of $K_6$ \cite{Godsil1}. Note that $L(K_6) \otimes J_2$ is a regular graph on $30$ vertices with spectrum $\left\lbrace \left[16\right]^1,\left[4\right]^5, \left[0\right]^{15}, \left[-4\right]^{9}\right\rbrace$. 
\end{exemplo}

\begin{exemplo} \label{reg4van}
Van Dam and Spence \cite{VanDam4} listed all \textit{feasible spectra} for connected regular graphs with four distinct eigenvalues and at most $30$ vertices. 
Inspecting their list and also the spectrum of the complement of the listed graphs we obtained all possible spectra for connected regular graphs in $\mathcal{G}_n$ with four distinct eigenvalues where the index is the only simple eigenvalue, for $n \leq 30$. 
They are presented in Table \ref{table2}, where $\#$ denotes the number of graphs with that spectrum given in \cite{VanDam4}. In each case, some graphs with that spectrum  are also listed. 

{\small
\begin{table}[H]
\def\arraystretch{1.3} % vertical stretch factor
  \begin{center}
    \begin{tabular}{|c|c|l|l|}% 
     \hline $n$ &  $\#$ & \textit{Spectrum} & \textit{Graphs \footnotemark} \\
      \hline
      $12$ & $2$ & $\left\lbrace\left[4\right]^1,\left[2\right]^3,\left[0\right]^3,\left[-2\right]^5\right\rbrace$ &   $L\left(Q_3\right)$ , \, $BCS_9$  \\
       \hline
       $12$ & $1$ & $\left\lbrace\left[6\right]^1,\left[2\right]^3,\left[0\right]^2,\left[-2\right]^6\right\rbrace$ & $L\left(CP\left(3\right)\right)$ \\
       \hline
        $18$ & $1$ & $\left\lbrace\left[12\right]^1,\left[3\right]^2,\left[0\right]^9,\left[-3\right]^6\right\rbrace$ & $\overline{K_{3,3} \square K_3}$ \\
       \hline
        $18$ & $2$ & $\left\lbrace\left[9\right]^1,\left[3\right]^3,\left[0\right]^8,\left[-3\right]^6\right\rbrace$ &  N/A
        \\
       \hline
       $24$ & $5$ & $\left\lbrace\left[8\right]^1,\left[4\right]^3,\left[0\right]^{15},\left[-4\right]^5\right\rbrace$ & 
       $L\left(Q_3\right)\otimes J_2$, \, $BCS_9\otimes J_2$        \\
       \hline
       $24$ & $28$ & $\left\lbrace\left[12\right]^1,\left[4\right]^3,\left[0\right]^{14},\left[-4\right]^6\right\rbrace$ &   
       $LCP\left(3\right)\otimes J_2$ 
       \\
       \hline
       $27$ & $4$ & $\left\lbrace\left[6\right]^1,\left[3\right]^6,\left[0\right]^{12},\left[-3\right]^8\right\rbrace$ &  $H(3,3)$ , \, $3$-cover $\left(C_3\otimes J_3\right)$ %\cite{VanDam1995}
       \\
       \hline
       $27$ & $13$ & $\left\lbrace\left[18\right]^1,\left[3\right]^6,\left[0\right]^{8},\left[-3\right]^{12}\right\rbrace$ &  $\overline{H(3,3)_3}$% Complement of graphs in row $99$ of table $A1$ in \cite{VanDam4} \
       \\
       \hline
       $27$ & $\geq 1$ & $\left\lbrace\left[12\right]^1,\left[3\right]^8,\left[0\right]^{6},\left[-3\right]^{12}\right\rbrace$ & $H(3,3)_2$ % complement of the graph in row $105$ of Table $A1$ in \cite{VanDam4} \
       \\
       \hline
       $30$ & $\geq 68876$ & $\left\lbrace\left[12\right]^1,\left[3\right]^{10},\left[0\right]^{5},\left[-3\right]^{14}\right\rbrace$ & $L_3\left( 6\right)\setminus6$-coclique \
       \\
       \hline
       $30$ & $\geq 1487$ & $\left\lbrace\left[16\right]^1,\left[4\right]^{5},\left[0\right]^{15},\left[-4\right]^{9}\right\rbrace$ & $ srg\left(15, 8, 4, 4\right) \otimes J_2 $
       
       \\
       \hline
       $30$ & $\geq 24931$ & $\left\lbrace\left[15\right]^1,\left[3\right]^{10},\left[0\right]^{4},\left[-3\right]^{15}\right\rbrace$ & $\overline{srg\left(35,16,6,\right)\setminus5\text{-clique}}$% complement of the graphs in row $170$ of Table $A1$ in \cite{VanDam4}
       \\
       \hline
      \end{tabular}
      \end{center}
       \caption{Spectra of connected regular graphs with four distinct \mbox{eigenvalues} in $\mathcal{G}_n$, $n\leq 30$, where the index is the only simple}
    \label{table2}
      \end{table}
  } 

\footnotetext{$G_1\square G_2$ is the \textit{Cartesian product} of $G_1$ and $G_2$; $BCS_9$ is the graph No. 9 in \cite[Table 9.1]{Bussemaker}; $CP(3)$ is the \textit{cocktail party graph of order $3$}; $H\left(3,3\right) \cong K_3 \square K_3 \square K_3$ is a Hamming graph; $G_i$ denotes the distance-$i$ graph of $G$.}
\end{exemplo}

From any graph in Table \ref{table2} we can construct an infinite family of connected regular graphs in $\mathcal{G}_{nm}$  with four distinct eigenvalues where the index is the only simple eigenvalue. For instance, 
 $L\left(Q_3\right)\otimes J_m$ and $BCS_9\otimes J_m$ are cospectral with spectrum $$\left\lbrace\left[4m\right]^1,\left[2m\right]^3,\left[0\right]^{12m-9},\left[-2m\right]^5\right\rbrace,$$ for all $m \geq 1$, and  $L\left(CP\left(3\right)\right)\otimes J_m$ is a regular graph in $\mathcal{G}_{nm}$ with spectrum 
 $$\left\lbrace\left[6m\right]^1,\left[2m\right]^3,\left[0\right]^{12m-10},\left[-2m\right]^6\right\rbrace.$$ Note that the graphs in these families with at most $30$ vertices appear in Table \ref{table2}.

%%%%%%%%%%%%%%%%%%%%%%%%%%%%%%%%%%%%%%%%%%%%%%%%%%
%%%%%%%%%%%%%%%%%%%%%%%%%%%%%%%%%%%%%%%%%%%%%%%%%%%
\section{Disconnected graphs}\label{sec4}

In the previous sections we considered connected graphs belonging to the family  $\mathcal{G}_n$, the class of nonempty graphs of order $n$ that satisfy the properties required in Problem \ref{prob}. Recall that all graphs in the family $\mathcal{H}_n$ of nonempty graphs of order $n$ that satisfy the properties required in Problem  \ref{prob2} are connected, and $\mathcal{G}_n \supset \mathcal{H}_n$. The following result characterizes the disconnected graphs in $\mathcal{G}_n$.

\begin{prop}\label{disc}
Let $G$ be disconnected graph of order $n$ with index $\lambda$. Then $G$ is in $\mathcal{G}_n$ if and only if one of the following cases holds.
\begin{itemize}
\item[(i)] The spectrum of $G$ is  $\left\lbrace\left[\lambda\right]^{\frac{n-t}{2}}, \left[0\right]^{t},\left[-\lambda\right]^{\frac{n-t}{2}}\right\rbrace$, with $t \geq 0$ and $ \lambda  \geq 1$. The connected components of $G$ are isolated vertices or complete bipartite graphs $K_{p,q}$ such that $pq = \lambda^2$.  In particular, $\lambda = 1$ if and only if every connected component of $G$ that is not an isolated vertex is the complete graph $K_2$.
\item[(ii)] $G$ is integral with spectrum $\left\lbrace\left[\lambda\right]^1,\left[\mu\right]^{n-k-t-1},\left[0\right]^{t},\left[-\mu\right]^{k}\right\rbrace$, with $ t \geq 0$, $k \geq 2$ and  $\lambda > \mu  \geq 1 $. Exactly one connected component $G_1$ of $G$ contains the index as an eigenvalue. In addition $G_1 \in \mathcal{G}_r$, for some $r<n$, and every other connected component of $G$ that is not an isolated vertex is a complete bipartite graph $K_{p,q}$ such that $pq=\mu^{2}$. In particular, $\mu =1$ if and only if $G_1$ is the complete graph $K_r$ and every other connected component of $G$ that is not an isolated vertex is the complete graph $K_2$. 
\end{itemize}
\end{prop}

\begin{proof} \renewcommand{\qedsymbol}{}
We divide the proof in two cases. First suppose that all nonzero eigenvalues of $G$ have the same absolute value. Then $G \in \mathcal{G}_n$ if and only if $Spec\left(G\right)=\left\lbrace\left[\lambda\right]^{\frac{n-t}{2}},\left[0\right]^{t},\left[-\lambda\right]^{\frac{n-t}{2}}\right\rbrace$, with $t\geq 0 $. Theorem \ref{teoimp2} implies that any connected component of $G$ that is not an isolated vertex is a complete bipartite graph. Since the spectrum of the complete bipartite graph $K_{p,q}$ is $\left\lbrace\left[\sqrt{pq}\right]^{1}, [0]^{p+q-2},\left[-\sqrt{pq}\right]^{1} \right\rbrace$, it follows that for any two connected components $K_{a,b}$ and $K_{c,d}$ of $G$ we have $ab=cd= \lambda^2$.
The case where $\lambda = 1$ follows from the fact that $K_{1, 1}$ is the complete graph $K_2$, which concludes the proof of $(i)$. 

Now suppose that $G$ has exactly two distinct nonzero absolute eigenvalues. Then \mbox{$G \in \mathcal{G}_n$} if and only if its spectrum is of the form $\left\lbrace\left[\lambda\right]^1,\left[\mu\right]^{n-k-t-1},\left[0\right]^{t},\left[-\mu\right]^{k}\right\rbrace$, where $ t \geq 0$, $k \geq 2$ and  $\lambda > \mu >0$. Note that the index $\lambda$ is a simple eigenvalue, otherwise $G \in \mathcal{G}_n$ would have only one nonzero absolute eigenvalue. Hence  only one connected component $G_1$ of $G$ contains $\lambda$ as an eigenvalue. It is clear that $G \in \mathcal{G}_n$ if and only if $G_1 \in \mathcal{G}_r$,  for some $r < n$,  $G_1$ also contains $-\mu$ as an eigenvalue, and all other connected components of $G$ are either isolated vertices or they have exactly two distinct nonzero eigenvalues: %with the same absolute value
$\mu$ and $-\mu$.  Thus, by Theorems \ref{teoimp3} and \ref{teoimp4}, 
the spectrum of $G_1$ is integral and consequently the spectrum of $G$ is integral. Besides, Theorem \ref{teoimp2} implies that each connected component of $G$ different from $G_1$ that is not an isolated vertex is a complete bipartite graph $K_{p,q}$ such that $pq = \mu^{2}$. In particular,  $\mu =1$ if and only if each connected component $K_{p, q}$ is the complete graph $K_{2}$. Also note  that in this case  $G_1$ is the complete graph $K_r$, otherwise it would have diameter greater than 1 and its smallest eigenvalue would be at most $-\sqrt{2}$, a contradiction.\QEDA 

\end{proof}

It follows from Proposition \ref{disc} that all disconnected graphs in $\mathcal{G}_n$ that do not have a connected component $K_{p, q}$ such that $pq$ is not a perfect square are integral. It also follows that to completely characterize a disconnected graph $G$ in $\mathcal{G}_n$ it is enough to characterize the connected component $G_1$ that contains the index of $G$. 

\section{Concluding remarks}
\label{open}

In this work we investigated Problems \ref{prob} and \ref{prob2} proposed by Nikiforov \cite{Niki2}. We \mbox{considered}  the families $\mathcal{G}_n$ and $\mathcal{H}_n$ of nonempty graphs that satisfy the properties required in Problems \ref{prob} and \ref{prob2}, respectively.

We gave a constructive characterization of all connected regular graphs in $\mathcal{G}_n$ with three distinct eigenvalues, and with four distinct eigenvalues where the index is not the only simple. We also presented several infinite families of connected graphs in  $\mathcal{G}_n$ with four distinct eigenvalues where the index is the only simple eigenvalue. 

In the case of connected irregular graphs in $\mathcal{G}_n$, we gave a constructive characterization when they have three distinct eigenvalues, one equals $0$, and when they have three nonzero distinct eigenvalues, one equals $2$.  All irregular graphs in $\mathcal{G}_n$ with three distinct nonzero eigenvalues, which are precisely the graphs in $\mathcal{H}_n$, were characterized as integral multiplicative graphs. Although this characterization is not constructive, we presented two infinite families of these graphs which include those with $2$ as an eigenvalue. Using a product construction with graphs of these families we generated two new infinite families of connected irregular graphs in $\mathcal{G}_n$ with four distinct eigenvalues.

%%%%%%%%%%%%%%%%%%%%%%%%%%%%%%%%%%%%%%%%%%%%%%%%%%%%%%%%%%%%%%%%%%%%%

\section*{Acknowledgments} The authors would like to thank Vladimir Nikiforov for his suggestions on a preliminary version of this manuscript. This work is part of the doctoral studies of Nelcy Arévalo, who thanks CAPES for their support. 

\bibliographystyle{abbrv}
\bibliography{sample}

\end{document}